\documentclass[12pt,reqno,a4paper]{amsart}
\usepackage{euscript}
\usepackage{amssymb}
\usepackage[dvipsnames]{xcolor}
\usepackage{color}
\theoremstyle{plain}
\numberwithin{equation}{section}
\newtheorem{theorem}{Theorem}
\newtheorem{proposition}{Proposition}
\newtheorem{lemma}{Lemma}
\newtheorem{corollary}{Corollary}
\newtheorem{definition}{Definition}
\theoremstyle{remark}
\newtheorem{remark}{Remark}
\newtheorem{example}{Example}
\renewcommand{\epsilon}{\varepsilon}
\renewcommand{\phi}{\varphi}

\DeclareMathOperator{\Ima}{Im}
\DeclareMathOperator{\Ker}{Ker}

\def\N{\mathbb{N}}
\def\cA{\EuScript{A}}
\def\R{\mathbb{R}}

\def\Id{\text{\rm Id}}
\def\Z{\mathbb{Z}}

\title [Smooth linearization of nonautonomous dynamics ]{Smooth linearization of nonautonomous dynamics under polynomial behaviour
}

\author{Lucas Backes}
\address[Lucas  Backes]
{Department of Mathematics\\
    Universidade Federal do Rio Grande do Sul\\
    Av. Bento Goncalves 9500, CEP 91509-900, Porto Alegre, RS, Brazil}
\email[L. Backes]{lucas.backes@ufrgs.br}

\author{ Davor Dragi\v cevi\'c}
\address[Davor Dragi\v cevi\'c]
{Faculty of Mathematics\\
    University of Rijeka\\
    Rijeka 51000, Croatia}
\email[D.~Dragi\v cevi\'c]{ddragicevic@math.uniri.hr}

\author{ Wenmeng Zhang}
\address[Wenmeng Zhang]
{School of Mathematical Sciences\\
    Chongqing Normal University\\
    Chongqing 401331, P.R.China}
\email[W.~M.~Zhang]{wmzhang@cqnu.edu.cn}

\date{}

\begin{document}
\maketitle

\begin{abstract}
The main purpose of this paper is to formulate new conditions for smooth linearization of nonautonomous systems with discrete and continuous time. Our results assume that the linear part admits a nonuniform polynomial dichotomy and that the associated polynomial dichotomy spectrum exhibits appropriate spectral gap and spectral band conditions. This is in sharp contrast to  most of the previous results in the literature which  assumed that the linear part admits an exponential dichotomy. Our techniques exploit the relationship between polynomial and exponential dichotomies via a suitable reparametrization of time.

\vskip 0.2cm


\vskip 0.2cm
{\bf AMS (2020) subject classification:} 37C15, 37D25
\end{abstract}

\baselineskip 15pt   
\parskip 10pt         

\thispagestyle{empty}
\setcounter{page}{1}
\allowdisplaybreaks[4]


\section{Introduction}

One of the most important notions in the theory of dynamical systems is that of (uniform) hyperbolicity. It plays a major part in many situations like in the study of stability, chaos and bifurcations. The counterpart of this notion  in the case of nonautonomous dynamical systems is given by the notion of \emph{exponential dichotomy}, whose study dates back to the pioneering work of Perron \cite{Per-30}. In the discrete time case, we say that the nonautonomous dynamics defined by the sequence $(A_n)_{n\in \Z}$ of invertible linear operators on $\R^d$ admits an exponential dichotomy
if there exist constants ${K},\lambda>0$ and a sequence of projections $(P_n)_{n\in \Z}$ on $\R^d$ such that for every $n,m\in \Z$ with $m>n$, $A_nP_n=P_{n+1}A_n$ and 
\begin{align}\label{ED-00}
\| A_{m-1}\cdots A_n P_n\| \le  {K}e^{-\lambda(m-n)},\quad
\| A^{-1}_n\cdots A^{-1}_{m-1} (\Id-P_m)\| \le  {K} e^{-\lambda (m-n)}.
\end{align}
As for hyperbolicity in the case of automomous dynamical systems, the notion of an exponential dichotomy also has wide applications like in the study of invariant manifolds, invariant foliations and normal forms \cite{BV-book,Pot-book,SY-book}, and there exists an extensive literature involving this notion for both finite and infinite dimensional systems (see for instance \cite{CL-JDE95,HZ-JDE04,P-JDE84,ZZ-JFA16}).

Even though the notion of an  exponential dichotomy is an important concept, there are many important classes of systems that do not fit in this framework. For example, Coppel~\cite[p.12]{Cop-book} gave an example 
of a planar nonautonomous system without bounded growth that does not admit an exponential dichotomy even though its restriction to the stable and unstable subspaces are  contractions and expansions, respectively. As observed in \cite[p. 683 and Example 2]{ZLZ-JDE17}, what fails to be satisfied for this example to admit an (uniform) exponential dichotomy is that the amount of time that we need to wait until we see actual contraction (or expansion) depends on the initial time. In terms of condition \eqref{ED-00} this means that the coefficient $K$ may depend on $n$ (or $m$) and that it can go to infinite whenever $n$ (or $m$) vary.
This gives rise to the notion of \emph{nonuniform} exponential dichotomy introduced by Barreira and Valls~\cite{BV-book}, which turns out to be ubiquitous in the context of the ergodic theory. 

On the other hand, due to the flexibility of the nonautonomous setting, it is easy to construct broad classes of systems which admit a splitting into stable and unstable directions, but with non-exponential rates of contraction and expansion. In the present paper we will be concerned with systems which combined both of the previous features: the rates of contraction and expansion along stable and unstable directions are nonuniform and non-exponential. Among many meaningful possibilities, we will consider the notion of a \emph{nonuniform polynomial dichotomy}, which was introduced independently (and in a slightly different manner) by Bento and Silva~\cite{BS1, BS2}
and Barreira and Valls~\cite{BV-NA09} (see also~\cite{BFVZ}). We recall that  a sequence   $(A_n)_{n\in \N}$ of invertible  linear operators on $\mathbb R^d$  admits a \emph{nonuniform polynomial dichotomy} if there exist ${K},\lambda>0$, $\varepsilon\geq 0$ and a sequence of projections $(P_n)_{n\in \N}$ on $\R^d$ such that for every $n,m\in \N$ with $m>n$, $A_nP_n=P_{n+1}A_n$ and
\begin{align*}
\| A_{m-1}\cdots A_n P_n\| \le  {K} (m/n)^{-\lambda}n^\varepsilon,
\quad
\| A^{-1}_n\cdots A^{-1}_{m-1} (\Id-P_m)\| \le  {K} (m/n)^{-\lambda}m^\varepsilon.
\end{align*}
As emphasized in~\cite{BFVZ}, this notion arises quite naturally under non-vanishing of appropriate polynomial Lyapunov exponents.

In what follows, we are interested in studying the problem of smooth linearization of  nonautonomous systems whose linear part admits a  nonuniform polynomial dichotomy. We will consider systems with both discrete and continuous time.  This problem consists of finding a time dependent smooth change of coordinates sending the solutions of the nonlinear system into the solutions of the linear one. 
In our main results, we first formulate smooth linearization results for one-sided systems whose linear part admits a nonuniform exponential dichotomy, which are  deduced from some existing results for two-sided systems \cite{DZZ,DZZ20}. Then, these results are applied to the setting of a nonuniform polynomial dichotomy to obtain smooth linearization results when the associated \emph{polynomial dichotomy spectrum} exhibits appropriate spectral gap and spectral band conditions.  Our techniques rely on the relationship between polynomial and exponential dichotomies via a suitable reparametrization of time. A key concept that is used in our results is given by the polynomial dichotomy spectrum and an important part of the paper is devoted to study the properties of this set. We emphasize that our results are in sharp contrast to  most of the previous results in the literature which assumed that the linear part admits an (uniform/nonuniform) exponential dichotomy. 

The linearization problem described above is one of the most fundamental ones in the theory of dynamical systems and has a long history (of which we do not attempt to give a complete survey). In the context of autonomous dynamics, the problem of formulating sufficient conditions under which the conjugacy exhibits higher regularity properties was first considered in the pioneering works of Sternberg~\cite{Sternberg2, Sternberg58} who proved in the 1950s that 
for each $r\in \mathbb{N}$ there is an integer $k\in \mathbb{N}$ such that $C^k$ hyperbolic diffeomorphisms with non-resonance conditions up to order $k$ admit a  $C^r$ linearization.
We also mention the fundamental works of Hartman \cite{HartPAMS60} and Grobman \cite{Grobman} who proved, independently, in the 1960s that $C^1$ hyperbolic diffeomorphisms can be $C^0$ linearized near the hyperbolic ﬁxed point. It shows that if one is only interested in the case where $r=0$ in the Sternberg's theorem, then one can take $k=1$ and, moreover, no non-resonance condition is needed (see Section \ref{sec: lin exp dich} for more details). We note that  infinite dimensional versions of this result were obtained independently by Palis (\cite{Palis}) and Pugh (\cite{Pugh}).

The first nonautonomous version of the Grobman-Hartman theorem was established by Palmer \cite{Palmer} for the case of continuous time dynamics. A discrete time version of Palmer's result was formulated in \cite{AW}. Both these results were obtained by assuming that the associated linear system admits a uniform exponential dichotomy.
The problem of smooth linearization for  nonautonomous systems was considered only recently. To the best of our knowledge, the ﬁrst results in this direction
were obtained in \cite{CMR,CR}, in the uniformly/nonuniformly exponentially stable case. Moreover, in \cite{CDS} the authors have established a Sternberg-type theorem in the setting of 
a uniform exponential dichotomy with continuous time.
More relevant to our context are the works \cite{DZZ,DZZ20} (see also~\cite{D1}) where $C^1$, differentiable (at 0) and H\"older linearization results were obtained, respectively, in the setting of a nonuniform exponential dichotomy under some spectral gap and spectral band conditions. These results will play a major part in our proofs, as we already emphasized. 

Finally, we mention  recent results \cite{BD, CJ, J} dealing with the smooth linearization of nonautonomous systems in the absence of  any kind of dichotomy, non-resonance or spectral gap assumptions. However, as observed in~\cite{BD},  such results in the setting of~\cite{DZZ} (and consequently also in the setting of the present paper)  can fail to be applicable or could yield weaker results.

This paper is organized as follows: we start our study focusing in discrete time dynamical systems (Sections \ref{sec: pol dich}-\ref{sec: inf dim discrete}). In Section \ref{sec: pol dich} we recall some notions of dichotomy and highlight some relations between them. Moreover, we introduce and study a notion of dichotomy spectrum for systems admitting a nonuniform polynomial dichotomy. In Section \ref{sec: lin exp dich} we present results about $C^1$, differentiable (at 0) and H\"older linearization under the assumption of nonuniform exponential dichotomy on the half line. Section \ref{sec: lin pol} is devoted to establish $C^1$, differentiable (at 0) and H\"older linearization results under the assumption of  nonuniform polynomial dichotomy. At the end of Section \ref{sec: lin pol}, we explain how the results from the previous section can be extended to the infinite dimensional setting. Finally, in Section \ref{sec: continuous time case} we obtain continuous time versions of the results obtained in Section \ref{sec: lin pol}. Appendix A is devoted to strengthen the properties of the conjugacies given in \cite{DZZ}, while in Appendix B we prove Proposition \ref{22}.

\section{Polynomial dichotomy and dichotomy spectrum} \label{sec: pol dich}
In this section we introduce several notions of dichotomies and dichotomy spectra and present some useful properties and relations between these notions. We start by recalling some definitions and introducing some notations. 

Let $\N$ denote the set of all natural numbers, $\Z$ the set of integers, $\Z^+$ the set of nonnegative integers and $\Z^-$ the set of nonpositive integers. Given $J\in \{\Z, \Z^-, \Z^+\}$ and a sequence $(A_n)_{n\in J}$ of invertible operators on $\R^d$, for $m,n\in J$ let us consider
\begin{equation}\label{eq: cocycle}
\cA(m,n):=\begin{cases}
A_{m-1} \ldots A_n, & m>n; \\
\Id, & m=n;\\
A_m^{-1} \ldots A_{n-1}^{-1}, & m<n,
\end{cases}
\end{equation}
where $\Id$ denotes the identity operator on $\R^d$.

\subsection{Nonuniform polynomial and exponential dichotomies}

We  introduce the notion of a nonuniform strong polynomial dichotomy.
\begin{definition}\label{def-SPD}
Let $(A_n)_{n\in \N}$ be a sequence of invertible operators on $\R^d$. We say that $(A_n)_{n\in \N}$ admits a \emph{nonuniform strong polynomial dichotomy} if there exist constants ${K}>0$,  $a\ge \lambda>0$, $\varepsilon\ge 0$ and a sequence of projections $(P_n)_{n\in \N}$ on $\R^d$ such that the following properties hold:
\begin{itemize}
\item for $n\in \N$,
\begin{equation}\label{pro}
A_nP_n=P_{n+1}A_n;
\end{equation}
\item for $m\ge n$, $m,n\in \N$,
\begin{equation}\label{pd1}
\| \cA(m,n)P_n\| \le {K}\bigg (\frac m n\bigg )^{-\lambda}n^\varepsilon,
\quad
\| \cA(n, m)Q_m\| \le {K} \bigg (\frac m n\bigg )^{-\lambda}m^\varepsilon,
\!\!\!\!\!\!\!\!
\end{equation}
where $Q_m:={\rm Id}-P_m$;
\item for $m\ge n$,
\begin{equation}\label{bg}
\| \cA(m,n)\| \le {K}\bigg (\frac m n\bigg )^{a}n^\varepsilon, \qquad \| \cA(n,m)\| \le {K}\bigg (\frac m n\bigg )^{a}m^\varepsilon.
\end{equation}
\end{itemize}
In particular, when $\varepsilon=0$ we say that $(A_n)_{n\in \N}$ admits a \emph{uniform strong polynomial dichotomy}.
\end{definition}

We also recall the notion of  a nonuniform strong exponential dichotomy.
\begin{definition}
Let $(A_n)_{n\in J}$ be a sequence of invertible operators on $\R^d$. We say that $(A_n)_{n\in J}$ admits a \emph{nonuniform strong exponential dichotomy}
if there exist constants ${K}>0$,  $a\ge \lambda>0$, $\varepsilon\ge 0$ and a sequence of projections $(P_n)_{n\in \N}$ on $\R^d$ such that the following properties hold:
\begin{itemize}
\item for $n, n+1\in J$, \eqref{pro} holds;
\item for $m\ge n$, $m,n\in J$,
\begin{align}\label{ED-1}
\| {\mathcal A}(m,n)P_n\| \le  {K}e^{-\lambda(m-n)+\varepsilon|n|},
\quad
\| {\mathcal A}(n, m) Q_m\| \le  {K} e^{-\lambda (m-n)+ \varepsilon|m|},
\!\!\!\!\!\!\!\!
\end{align}
where $Q_m:=\Id-P_m$;
\item for $m\ge n$,
\begin{equation}\label{be}
\|{\mathcal A}(m,n)\| \le {K}e^{a(m-n)+\varepsilon |n|},\qquad\|{\mathcal A}(n,m)\| \le {K}e^{a(m-n)+\varepsilon|m|}.
\end{equation}
\end{itemize}
In particular, when $\varepsilon=0$ we say that $(A_n)_{n\in \N}$ admits a \emph{uniform strong exponential dichotomy}.
\end{definition}

\subsection{Polynomial and exponential dichotomies with respect to a sequence of norms}

\begin{definition}\label{def-SPD-SN}
Let $(\|\cdot\|_n)_{n\in \N}$ be a sequence of norms on $\R^d$.
 We say that a sequence $(A_n)_{n\in \N}$ of invertible operators on $\R^d$ admits a \emph{strong polynomial dichotomy with respect to} $(\|\cdot\|_n)_{n\in \N}$,
 if there exist constants  ${K}>0$,  $a\ge \lambda>0$ and a sequence of projections $(P_n)_{n\in \N}$ on $\R^d$ such that the following properties hold:
\begin{itemize}
\item for $n\in \N$, \eqref{pro} holds;
\item for $m\ge n$, $m,n\in \N$ and $x\in \R^d$,
\begin{equation}\label{pd1sn}
\begin{split}
&\| \cA(m,n)P_nx\|_m \le {K}\bigg (\frac m n\bigg )^{-\lambda}\|x\|_n,
\\
&\| \cA(n, m)Q_mx\|_n \le {K} \bigg (\frac m n\bigg )^{-\lambda}\|x\|_m,
\end{split}
\end{equation}
where $Q_m:={\rm Id}-P_m$;
\item  for $m\ge n$ and $x\in \R^d$,
\begin{equation}\label{bg1sn}
\| \cA(m,n)x\|_m \le {K}\bigg (\frac m n\bigg )^{a}\|x\|_n, \quad \| \cA(n,m)x\|_n \le {K}\bigg (\frac m n\bigg )^{a} \|x\|_m.
\!\!\!
\end{equation}
\end{itemize}
\end{definition}
\begin{remark}
Observe that in the above Definitions~\ref{def-SPD}-\ref{def-SPD-SN}, it is allowed that $P_n=\Id$ or $P_n={\bf 0}$ (the nul-operator), which correspond to the contractive case and the expansive case, respectively.
\end{remark}

The following result gives the relationship between nonuniform strong polynomial dichotomies and polynomial dichotomies with respect to a sequence of norms.
\begin{proposition}\label{LN}
Let $(A_n)_{n\in \N}$ be a sequence of invertible operators on $\R^d$. The following properties are equivalent:
\begin{enumerate}
\item $(A_n)_{n\in \N}$ admits a nonuniform strong polynomial dichotomy;
\item $(A_n)_{n\in \N}$ admits a strong polynomial dichotomy with respect to $(\|\cdot\|_n)_{n\in \N}$ with the property that there exist $C>0$ and $\delta \ge 0$ such that
\begin{equation}\label{ln1}
\|x\| \le \|x\|_n \le Cn^\delta \|x\|, \quad \text{for $x\in \R^d$ and $n\in \N$.}
\end{equation}
\end{enumerate}
\end{proposition}

\begin{proof}
Assume that $(A_n)_{n\in \Z}$ admits a strong polynomial dichotomy with respect to $(\|\cdot\|_n)_{n\in \N}$ such that~\eqref{ln1} holds with $C>0$ and $\delta \ge 0$. Observe that~\eqref{pd1sn} and~\eqref{ln1} imply that
\[
\| \cA(m,n)P_nx\| \le \| \cA(m, n)P_nx \|_m \le {K}\bigg (\frac m n\bigg )^{-\lambda}\|x\|_n \le C{K}\bigg (\frac m n\bigg )^{-\lambda}n^\delta \|x\|,
\]
and
\[
\| \cA(n, m)Q_mx\| \le \| \cA(n, m)Q_mx\| _n \le {K} \bigg (\frac m n\bigg )^{-\lambda}\|x\|_m \le C{K}\bigg (\frac m n\bigg )^{-\lambda}m^\delta \|x\|
\]
for $m\ge n$ and $x\in \R^d$. Simialrly, \eqref{bg1sn} and~\eqref{ln1} give that
\[
\| \cA(m,n)x\| \le C{K}\bigg (\frac m n\bigg )^{a}n^\delta \|x\|,\quad \| \cA(n,m)x\| \le C{K}\bigg (\frac m n\bigg )^{a}m^\delta \|x\|
\]
for $m\ge n$ and $x\in \R^d$. Thus, we conclude that $(A_n)_{n\in \N}$ admits a nonuniform strong polynomial dichotomy.

Suppose now that $(A_n)_{n\in \N}$ admits a nonuniform strong polynomial dichotomy.  For $n\in \N$ and $x\in \R^d$, we define
\[
\|x\|_n:=\|x\|_n^s +\|x\|_n^u,
\]
where
\[
\|x\|_n^s:=\sup_{m\ge n} \bigg(\| \cA(m,n)P_nx\| \bigg (\frac m n\bigg )^{\lambda} \bigg )+\sup_{m<n}\bigg (\| \cA(m,n)P_nx\| \bigg (\frac n m\bigg )^{-a} \bigg ),
\]
and
\[
\|x\|_n^u:=\sup_{m< n}\bigg (\| \cA(m,n)Q_nx\| \bigg (\frac n m\bigg )^{\lambda} \bigg )+\sup_{m\ge n}\bigg (\| \cA(m,n)Q_n x\| \bigg (\frac m n\bigg )^{-a} \bigg ).
\]
Observe that
\[
\|x\| \le \|P_nx\|+\|Q_nx\| \le \|x\|_n^s+\|x\|_n^u=\|x\|_n,
\]
for $n\in \N$ and $x\in X$. Moreover, \eqref{pd1} and~\eqref{bg} imply that
\[
\|x\|_n^s\le Kn^\varepsilon \|x\| +Kn^\varepsilon \|P_nx\|\le ({K}+{K}^2)n^{2\varepsilon} \|x\|,
\]
and similarly,
$
\|x\|_n^u \le ({K}+{K}^2)n^{2\varepsilon} \|x\|
$
for $x\in \R^d$ and $n\in \N$. We conclude that~\eqref{ln1} holds with $C=2({K}+{K}^2)>0$ and $\delta=2\varepsilon \ge 0$.

Moreover, for $m\ge n$   and $x\in \R^d$ we have that
\begin{equation}\label{x1}
\begin{split}
&\| \cA(m,n)P_n x\|_m = \| \cA(m,n)P_n x\|^s_m
\\
&=\sup_{k\ge m}\bigg(\| \cA(k,n)P_nx\| \bigg (\frac k m\bigg )^{\lambda} \bigg )+\sup_{k<m}\bigg (\| \cA(k,n)P_nx\| \bigg (\frac m k\bigg )^{-a} \bigg ) \\
&\le \sup_{k\ge m}\bigg(\| \cA(k,n)P_nx\| \bigg (\frac k m\bigg )^{\lambda} \bigg )+\sup_{n\le k<m}\bigg(\| \cA(k,n)P_nx\| \bigg (\frac k m\bigg )^{\lambda} \bigg )\\
&\phantom{\le}+\sup_{k<n}\bigg (\| \cA(k,n)P_nx\| \bigg (\frac m k\bigg )^{-a} \bigg ) \\
&\le 2\sup_{k\ge n}\bigg(\| \cA(k,n)P_nx\| \bigg (\frac k m\bigg )^{\lambda} \bigg )+\sup_{k<n}\bigg (\| \cA(k,n)P_nx\| \bigg (\frac m k\bigg )^{-a} \bigg ) \\
&= 2\bigg (\frac m n\bigg )^{-\lambda}\sup_{k\ge n}\bigg(\| \cA(k,n)P_nx\| \bigg (\frac k n\bigg )^{\lambda} \bigg ) \\
&\phantom{=}+\bigg (\frac m n\bigg )^{-a}\sup_{k<n}\bigg (\| \cA(k,n)P_nx\| \bigg (\frac n k\bigg )^{-a} \bigg ), \\
&\le 2\bigg (\frac m n\bigg )^{-\lambda} \|x\|_n,
\end{split}
\end{equation}
establishing the first inequality in~\eqref{pd1sn}. On the other hand, for $m\le n$ and $x\in \R^d$ we have that
\begin{equation}\label{x2}
\begin{split}
&\| \cA(m,n)P_n x\|_m = \| \cA(m,n)P_n x\|^u_m
\\
&\le \sup_{k\ge n}\bigg(\| \cA(k,n)P_nx\| \bigg (\frac k m\bigg )^{\lambda} \bigg )+\sup_{m\le k<n}\bigg(\| \cA(k,n)P_nx\| \bigg (\frac m k\bigg )^{-a} \bigg )\\
&\phantom{\le}+\sup_{k<m}\bigg (\| \cA(k,n)P_nx\| \bigg (\frac m k\bigg )^{-a} \bigg ) \\
&\le \sup_{k\ge n}\bigg(\| \cA(k,n)P_nx\| \bigg (\frac k m\bigg )^{\lambda} \bigg )+2\sup_{k<n}\bigg (\| \cA(k,n)P_nx\| \bigg (\frac m k\bigg )^{-a} \bigg )\\
&=\bigg (\frac n m\bigg )^{\lambda} \sup_{k\ge n}\bigg(\| \cA(k,n)P_nx\| \bigg (\frac k n\bigg )^{\lambda} \bigg )\\
&\phantom{=}+2 \bigg (\frac n m\bigg )^a\sup_{k<n} \bigg (\| \cA(k,n)P_nx\| \bigg (\frac n k\bigg )^{-a} \bigg )\\
&\le 2\bigg (\frac n m\bigg )^a \|x\|_n.
\end{split}
\end{equation}
Similarly, one can establish that
\begin{equation}\label{x3}
\| \cA(m,n)Q_n x\|_m  \le 2\bigg (\frac n m\bigg )^{-\lambda} \|x\|_n \quad \text{for $x\in \R^d$ and $m\le n$,}
\end{equation}
and
\begin{equation}\label{x4}
\|\cA(m,n)Q_n x\|_m \le 2\bigg (\frac m n\bigg )^a \|x\|_n, \quad \text{for $x\in \R^d$ and $m\ge n$.}
\end{equation}
In particular, the second estimate in~\eqref{pd1sn} follows from~\eqref{x3}. Moreover, \eqref{x1} and~\eqref{x4} imply that
\[
\begin{split}
\|\cA(m,n)x\|_m  &\le \|\cA(m,n)P_nx\|_m+\|\cA(m,n)Q_nx\|_m \\
&\le 2\bigg (\frac m n\bigg )^{-\lambda} \|x\|_n+2\bigg (\frac m n\bigg )^a \|x\|_n 
\le 4\bigg (\frac m n\bigg )^a \|x\|_n,
\end{split}
\]
for $m\ge n$ and $x\in \R^d$. Hence, the first inequality in~\eqref{bg1sn} holds. Similarly, using~\eqref{x2} and~\eqref{x3} one can establish the second estimate in~\eqref{bg1sn}. The proof is completed.
\end{proof}

\begin{remark}
We note that a result similar to Proposition~\ref{LN} was established in~\cite[Proposition 7]{D}. However, the notion of nonuniform strong polynomial dichotomy considered in~\cite{D} is slightly different from the one considered in the present paper. More precisely, the notion introduced in~\cite{D} requires only that the first inequality in~\eqref{bg} holds.
\end{remark}

\begin{definition}\label{1753}
Let $J\in \{\Z, \Z^+, \Z^- \}$ and
$(\|\cdot\|_n)_{n\in J}$ be a sequence of norms on $\R^d$.
 We say that a sequence $(A_n)_{n\in J}$ of invertible operators on $\R^d$ admits a \emph{strong exponential dichotomy with respect to} $(\|\cdot\|_n)_{n\in J}$,
 if there exist constants  ${K}>0$,  $a\ge \lambda>0$ and a sequence of projections $(P_n)_{n\in J}$ on $\R^d$ such that the following properties hold:
\begin{itemize}
\item for $n, n+1\in J$, \eqref{pro} holds;
\item for $m\ge n$, $m,n\in J$ and $x\in \R^d$,
\begin{equation}\label{ed1sn}
\begin{split}
&\| \cA(m,n)P_nx\|_m \le {K}e^{-\lambda (m-n)}\|x\|_n,
\\
&\| \cA(n, m)Q_mx\|_n \le {K}e^{-\lambda (m-n)}\|x\|_m,
\end{split}
\end{equation}
where $Q_m:={\rm Id}-P_m$;
\item  for $m\ge n$, $m, n\in J$ and $x\in \R^d$,
\begin{equation}\label{bg1sne}
\| \cA(m,n)x\|_m \le {K}e^{a(m-n)}\|x\|_n, \quad \| \cA(n,m)x\|_n \le {K}e^{a(m-n)} \|x\|_m.
\!\!\!\!\!\!\!\!\!\!\!\!
\end{equation}
\end{itemize}
\end{definition}
We have the following version of Proposition~\ref{LN} for exponential dichotomies (see~\cite[Proposition 12]{BDV1}).
\begin{proposition}\label{LN2}
Let $J\in \{\Z, \Z^+, \Z^-\}$ and
$(A_n)_{n\in J}$ be a sequence of invertible operators on $\R^d$. The following properties are equivalent:
\begin{enumerate}
\item $(A_n)_{n\in J}$ admits a nonuniform strong exponential dichotomy dichotomy;
\item $(A_n)_{n\in J}$ admits a strong exponential dichotomy with respect to $(\|\cdot \|_n)_{n\in J}$ with the property that there exist $C>0$ and $\delta \ge 0$ such that
\begin{equation}\label{ln2}
\|x\| \le \|x\|_n \le Ce^{\delta |n|} \|x\|, \quad \text{for $x\in \R^d$ and $n\in J$.}
\end{equation}
\end{enumerate}
\end{proposition}

The following result establishes the relationship between exponential and polynomial dichotomies with respect to a sequence of norms.
\begin{proposition}\label{WW}
Let
$(\|\cdot\|_n)_{n\in \N}$ be a sequence of norms on $\R^d$.
Suppose that $(A_n)_{n\in \N}$ is a sequence of invertible operators on $\R^d$ with the property that there exist ${K}, a>0$ such that~\eqref{bg1sn} holds, and set
\begin{equation}\label{BN}
B_n:=\cA(2^{n+1}, 2^n), \quad n\ge 0.
\end{equation}
The following properties are equivalent:
\begin{itemize}
\item $(A_n)_{n\in \N}$ admits a strong polynomial dichotomy with respect to $(\|\cdot\|_n)_{n\in \N}$;
\item $(B_n)_{n\in \Z^+}$ admits a strong exponential dichotomy with respect to $(\|\cdot\|_{2^n})_{n\in \Z^+}$.
\end{itemize}
\end{proposition}

\begin{proof}
Suppose that $(A_n)_{n\in \N}$ admits a  strong polynomial dichotomy with respect $(\|\cdot\|_n)_{n\in \N}$. It follows from~\eqref{pro} that
\[
P_{2^{n+1}}B_n=P_{2^{n+1}} \cA(2^{n+1}, 2^n)=\cA(2^{n+1}, 2^n)P_{2^n}=B_nP_{2^n}, \quad n\in \Z^+.
\]
By~\eqref{pd1sn}, we have that
\[
\| \mathcal B(m,n)P_{2^n}x\|_{2^m}=\| \cA(2^m, 2^n)P_{2^n} x\|_{2^m} \le {K}e^{-\lambda \log 2 (m-n)}\|x\|_{2^n},
\]
for $m\ge n$ and $x\in \R^d$. Similarly,
\[
\| \mathcal B(m, n)Q_{2^n}x\|_{2^m} \le {K}e^{-\lambda \log 2 (n-m)}\|x\|_{2^n},
\]
for $m\le n$ and $x\in \R^d$, where $Q_{2^n}:=\Id-P_{2^n}$. Finally, \eqref{bg1sn} implies that
\begin{equation}\label{te}
\begin{split}
&\|\mathcal B(m,n)x\|_{2^m} \le {K} e^{a\log 2 (m-n)} \|x\|_{2^m},
\\
&\| \mathcal B(n,m)x\|_{2^n}\le {K}e^{a\log 2 (m-n)}\|x\|_{2^m},
\end{split}
\end{equation}
for $m\ge n$ and $x\in \R^d$. We conclude that $(B_n)_{n\in \Z^+}$ admits a strong exponential dichotomy with respect to $(\|\cdot\|_{2^n})_{n\in \Z^+}$.

Assume now that $(B_n)_{n\in \Z^+}$ admits a strong exponential dichotomy with respect to $(\|\cdot\|_{2^n})_{n\in \Z^+}$ and projections $\tilde P_n$, $n\in \Z^+$. Namely, there exist ${K}', \lambda >0$ such that
\begin{equation}\label{z1}
\| \mathcal B(m,n)\tilde P_n x\|_{2^m}\le {K}'e^{-\lambda (m-n)} \|x\|_{2^n}, \quad \text{for $m\ge n$ and $x\in \R^d$,}
\end{equation}
\begin{equation}\label{z2}
\| \mathcal B(m,n)\tilde Q_n x\|_{2^m}\le {K}'e^{-\lambda (n-m)} \|x\|_{2^n}, \quad \text{for $m\le n$ and $x\in \R^d$,}
\end{equation}
where $\tilde Q_n:=\Id-\tilde P_n$,
and
\begin{equation}\label{pro2}
\tilde P_{n+1}B_n=B_n\tilde P_n, \quad n\in \Z^+.
\end{equation}
Take an arbitrary $k\in \N$. Then, there exists a unique $n\in \Z^+$ such that $2^{n} \le k<2^{n+1}$, which enables us to set
\[
P_k:=\cA(k,2^{n})\tilde P_n\cA(2^{n}, k).
\]
Observe that $P_k$ is a projection for each $k\in \N$.
We claim that~\eqref{pro} holds.  In fact, when $2^n\le k<k+1<2^{n+1}$, we have
\begin{align*}
P_{k+1}A_k
&=\cA(k+1,2^{n})\tilde P_n\cA(2^{n}, k+1)A_k
\\
&=A_k\cA(k,2^{n})\tilde P_n\cA(2^{n}, k)=A_k P_k
\end{align*}
and, when $2^n\le k<k+1 = 2^{n+1}$, we have (using~\eqref{pro2}) that 
\begin{align*}
P_{k+1}A_k=\tilde P_{n+1}A_k
&=\cA(k+1,2^{n+1})\tilde P_{n+1}\cA(2^{n+1}, k+1)A_k
\\
&=A_k\cA(k,2^{n+1})\tilde P_{n+1}\cA(2^{n+1}, k)
\\
&= A_k\cA(k,2^{n})\cA(2^{n},2^{n+1})\tilde P_{n+1}\cA(2^{n+1}, 2^{n})\cA(2^{n}, k)
\\
&=A_k\cA(k,2^{n})B_n^{-1}\tilde P_{n+1}B_n\cA(2^{n}, k)
\\
&=A_k\cA(k,2^{n})\tilde P_n\cA(2^{n}, k)=A_kP_{k}.
\end{align*}
Take now arbitrary $k,l \in \N$ such that $k\ge l$, and choose $m, n\in \Z^+$ such that $2^{m} \le k<2^{m+1}$ and $2^{n} \le l<2^{n+1}$. Obviously, $m\ge n\ge 0$. Moreover, observe that
\begin{align}\label{aap}
\cA(k,l)P_l=\cA(k, 2^{n})\tilde P_n\cA(2^{n}, l)=\cA(k, 2^{m})\mathcal B(m,n)\tilde P_n\cA(2^{n}, l).
\end{align}
Hence, by~\eqref{bg1sn} and~\eqref{z1} we have that
\[
\begin{split}
\|\cA(k,l)P_lx\|_k &=\|\cA(k, 2^{m})\mathcal B(m,n)\tilde P_n\cA(2^{n}, l)x\|_k \\
&\le {K}\bigg (\frac{k}{2^m} \bigg )^a \| \mathcal B(m,n)\tilde P_n\cA(2^{n}, l)x\|_{2^m}\\
&\le KK'2^ae^{-\lambda (m-n)} \|\cA(2^{n}, l)x\|_{2^n} \\
&\le KK'2^ae^{-\lambda (m-n)} {K}\bigg (\frac{l}{2^n} \bigg )^a \|x\|_l\\
&\le {K}^2{K}'4^ae^{-\lambda (m-n)} \|x\|_l 
={K}^2{K}'4^a e^{-\lambda \left( \lfloor \frac{\log k}{\log 2}\rfloor - \lfloor \frac{\log l}{\log 2} \rfloor \right)}\|x\|_l \\
&\le  {K}^2{K}'4^ae^\lambda \bigg (\frac k l\bigg )^{-\frac{\lambda}{\log 2}} \|x\|_l,
\end{split}
\]
for $k\ge l$ and $x\in \R^d$. This proves the first inequality in~\eqref{pd1sn}. Similarly, one can establish the second inequality in~\eqref{pd1sn}. Then, \eqref{pd1sn} and~\eqref{bg1sn} imply that
$(A_n)_{n\in \N}$ admits a  strong polynomial dichotomy with respect to $(\|\cdot\|_n)_{n\in \N}$. Thus, the  proof is completed.
\end{proof}

\begin{remark}
We observe that a result similar to Proposition~\ref{WW} was formulated in~\cite[Theorem 3.1]{DSS}. More precisely, \cite[Theorem 3.1]{DSS} does not require that~\eqref{bg1sn} holds and it deals with not necessarily strong polynomial and exponential dichotomies with respect to a sequence of norms.
\end{remark}

\subsection{Dichotomy spectra}

\begin{definition}\label{def-01}
Let $\mathbb A=(A_n)_{n\in \N}$ be a sequence of invertible operators and $\mathcal S:=\{ \|\cdot \|_n; \ n\in \N\}$ be a sequence of norms on $\R^d$. We define
$\Sigma_{PD, \mathbb A, \mathcal S}$ to be the set of all $\tau\in \R$ with the property that $\big (\big (\frac{n+1}{n}\big )^{-\tau}A_n\big )_{n\in \N}$ does not admit a strong polynomial dichotomy with respect to $\mathcal S$.
\end{definition}

\begin{definition}\label{def-SED}
Let $\mathbb A=(A_n)_{n\in \Z^+}$ be a sequence of invertible operators and $\mathcal S:=\{ \|\cdot \|_n; \ n\in \Z^+\}$ be a sequence of norms on $\R^d$. We define
$\Sigma_{ED, \mathbb A, \mathcal S}$ to be the set of all $\tau>0$ with the property that $\big (\frac 1 \tau A_n\big)_{n\in \Z^+}$ does not admit a strong exponential dichotomy with respect to $\mathcal S$.
\end{definition}

\begin{remark}\label{aaa}
Similarly, we can introduce $\Sigma_{ED, \mathbb A, \mathcal S}$ for two-sided sequences of invertible operators $(A_n)_{n\in \Z}$ and norms $\mathcal S=\{ \| \cdot \|_n; \ n\in \Z\}$.
\end{remark}

\begin{proposition}\label{22}
Let $\mathbb{A}=(A_n)_{n\in \Z^+}$ be a sequence of invertible operators and $\mathcal S:=\{ \|\cdot \|_n; \ n\in \Z^+\}$ be a sequence of norms on $\R^d$ with the property that there exist ${K}, a>0$ such that~\eqref{bg1sne} holds. Then, there exist $1\le r\le d$ and
$
a_1\le b_1 <\ldots <a_r \le b_r,
$
such that
\[
\Sigma_{ED, \mathbb A, \mathcal S}=\bigcup_{i=1}^r [a_i, b_i].
\]
\end{proposition}
The proof of the this proposition will be given in Section~\ref{A-B}.

\begin{corollary}\label{KOR}
Let $\mathbb A=(A_n)_{n\in \N}$ be a sequence of invertible operators and $\mathcal S:=\{ \|\cdot \|_n; \ n\in \N\}$ be a sequence of norms on $\R^d$ with the property that there exist ${K}, a>0$ such that~\eqref{bg1sn} holds.  Furthermore, let $\mathbb B=(B_n)_{n\in \Z^+}$ be given by~\eqref{BN} and set $\tilde{\mathcal S}:=\{ \|\cdot \|_{2^n}; \ n\in \Z^+\}$. Then,
\[
\tau\in \Sigma_{PD, \mathbb A, \mathcal S}\iff 2^\tau \in \Sigma_{ED, \mathbb B, \tilde{\mathcal S}}.
\]

\end{corollary}

\begin{proof}
It follows from Proposition~\ref{WW} that $\big(  (\frac{n+1}{n})^{-\tau}A_n \big )_{n\in \N}$ admits a polynomial dichotomy with respect to $\mathcal S$
 if and only if $(2^{-\tau}B_n)_{n\in \Z^+}$ admits an exponential dichotomy with respect to $\tilde{\mathcal S}$.
 Hence, $\tau\notin \Sigma_{PD, \mathbb A, \mathcal S}$ if and only if $2^\tau\notin \Sigma_{ED, \mathbb B, \tilde{\mathcal S}}$, which immediately implies the desired conclusion.
\end{proof}

\begin{corollary}\label{KOR1}
Let $\mathbb A=(A_n)_{n\in \N}$, $\mathbb B=(B_n)_{n\in \Z^+}$, $\mathcal S$ and $\tilde{\mathcal S}$ be as in the statement of Corollary~{\rm \ref{KOR}}. Then, there exist $1\le r\le d$ and
$
a_1\le b_1 <\ldots <a_r \le b_r
$
such that
\begin{align}\label{EDB}
\Sigma_{ED, \mathbb B,  \tilde{\mathcal S}}=\bigcup_{i=1}^r [a_i, b_i]
\end{align}
and
\begin{align}\label{PDA}
\Sigma_{PD, \mathbb A, \mathcal S}=\bigcup_{i=1}^r \left[\frac{\log a_i}{\log 2}, \frac{\log b_i}{\log 2}\right].
\end{align}
\end{corollary}

\begin{proof}
By~\eqref{bg1sn} we have that~\eqref{te} holds
for $m\ge n$ and $x\in \R^d$.  The desired conclusion now follows readily from Proposition~\ref{22} and Corollary~\ref{KOR}.
\end{proof}

\section{Smooth linearization under exponential dichotomy} \label{sec: lin exp dich}
We emphasize that  $C^0$ conjugacy  in general  does not preserve important dynamical properties
such as  characteristic directions and  smoothness of invariant manifolds. Preserving these properties needs at least $C^1$ conjugacy. In 1970s, an effort was made by Belitskii in~\cite{Bel73} to establish a $C^1$ linearization
of $C^2$ hyperbolic diffeomorphisms on $\mathbb{R}^d$ under certain second order non-resonance conditions. This result requires much weaker conditions on smoothness and non-resonance than the Sternberg's theorem in $C^1$ case. Moreover, Hartman's example given in~\cite{HartPAMS60} shows that Belitskii's non-resonance conditions cannot be avoided.

Since 1990s, two independent research directions emerged. In one direction, the  goal was to extend Belitskii's result to the infinite-dimensional setting, while the other direction was concerned with   differentiable   (at the fixed point $0$) linearization in the absence of non-resonance conditions. 
Concerning the first direction of research, we mention the work \cite{ZZJ}, which established  $C^1$ linearization result on  Banach spaces under appropriate  spectral gap and  spectral band conditions for the linear part. This result was further extended to the nonautonomous setting in \cite{DZZ}. In the second direction, van Strien~\cite{Stri-JDE90} claimed  the linearization result for  $C^2$ diffeomorphisms on $\R^d$ without any non-resonance conditions. The conclusion was that there exists  a conjugacy  which is simultaneously differentiable at 0 and H\"older
continuous near 0 (weaker than $C^1$ smoothness). However,  van Strien's proof was found to be incorrect (see~\cite{Ray-JDE98}). In~\cite{DZZ20}, the authors gave a correct proof of 
van Strien's result by using  different methods.

In what follows, we present two linearization results for nonautonomous dynamics under the assumption that the linear part admits an exponential dichotomy with respect to a sequence of norms. The first result (see Theorem~\ref{t0}) gives conditions under which the conjugacies are $C^1$, while the second result (see Theorem~\ref{t1}) is concerned with the case when the conjugacies are differentiable at $0$ and locally H\"older
continuous. In contrast to the existing result in the literature, the novelty is that Theorems~\ref{t0} and~\ref{t1} are concerned with one-sided dynamics (which will make them applicable to the problem of linearization under polynomial behaviour).

\subsection{$C^1$ linearization under exponential dichotomy}

\begin{theorem}\label{t0}
Let $\mathbb B=(B_n)_{n\in \Z^+}$ be a sequence of invertible operators on $\R^d$ that admits a strong exponential dichotomy with respect to a sequence of norms $\mathcal S=\{ \|\cdot \|_n; \ n\in \Z^+\}$.
Suppose that $\Sigma_{ED, \mathbb B, \mathcal S}$  has the form \eqref{EDB}, where
\begin{equation}\label{sp1}
a_1\le b_1 <\ldots a_k\le b_k <1 <a_{k+1} \le b_{k+1}<\ldots a_r \le b_r
\end{equation}
and
\begin{equation}\label{sp2}
\begin{cases}
a_{k+1}/b_k >\max \{b_r, a_1^{-1}\}, \\
b_i/a_i<b_k^{-1}, \ \forall i=1, \ldots, k, \ b_j/a_j < a_{k+1}, \forall j=k+1, \ldots, r.
\end{cases}
\end{equation}
Moreover, let $f_n\colon \R^d \to \R^d$, $n\in \Z^+$, be a sequence of $C^1$ maps such that
\begin{equation}\label{fzero}
f_n(0)=0 \quad \text{and} \quad Df_n(0)=0,
\end{equation}
and
\begin{equation}\label{fm}
\|Df_n(x)v\|_{n+1} \le \eta \|v\|_n, \quad \|(Df_n(x)-Df_n(y))v\|_{n+1} \le L\|x-y\|_n \cdot \|v\|_n,
\end{equation}
with constants $L, \eta>0$ for $x, y, v\in \R^d$. Then, provided that $\eta$ is sufficiently small,  there exists a sequence of $C^1$-diffeomorphisms $h_n\colon \R^d\to \R^d$, $n\in \Z^+$,  such that
\begin{itemize}
\item for $n\in \Z^+$,
$
h_{n+1} \circ (B_n+f_n)=B_n\circ h_n;
$
\item there exist $M, \rho>0$ such that
\begin{equation*}
\|D h_n(x)v\|_n\le M\|v\|_n, \quad \|Dh_n^{-1}(x)v\|_n \le M\|v\|_n, \quad \forall n\in \Z^+,
\end{equation*}
for all $v\in \R^d$ and $x,y\in \mathbb{R}^d$ satisfying $\|x\|_n\le \rho$.
\end{itemize}
\end{theorem}

\begin{remark}
Theorem~\ref{t0} is a consequence of Theorem \ref{Ap} whose proof will be given in Appendix A. In fact, Theorem \ref{Ap} deals with the case of $n\in \Z$, while Theorem \ref{t0} deals with the case of $n\in \Z^+$. Moreover, observe that $1\notin \Sigma_{ED, \mathbb B, \mathcal S}$ since $\mathbb B$ admits an exponential dichotomy with respect to the sequence of norms $\mathcal S$.
\end{remark}

\begin{proof}
Set $f_n:=0$ for all $n<0$. Then, \eqref{fm} holds for every $x, y\in \R^d$ and $n\in \Z$. Choose $c_1, \ldots, c_{r+1} \in \R$ such that
\[
0<c_1<a_1\le b_1<c_2 <a_2\le b_2 <\ldots a_r \le b_r <c_{r+1}.
\]
For each $i\in \{1, 2, \ldots, r+1\}$, set
\[
S_i:=\bigg \{ x\in \R^d: \sup_{n\ge 0} \frac{1}{c_i^n} \| \mathcal B(n,0)x\|<+\infty \bigg \}.
\]
Then (see the proof of Proposition~\ref{22}), for each $i\in \{1,2 \ldots, r+1\}$, the sequence $(\frac{1}{c_i}B_n)_{n\in \Z^+}$ admits a nonuniform exponential dichotomy with respect to projections $\tilde P_n^i$ for $n\in \Z^+$, where $\Ima \tilde P_0^i=S_i$. Moreover, the following holds true:
\begin{itemize}
\item $S_i$ does not depend on the particular choice of $c_i$;
\item $S_1=\{0\} \subsetneq S_2\ldots \subsetneq S_r \subsetneq S_{r+1}=\R^d$.
\end{itemize}
Next, choose subspaces $V_2, \ldots, V_{r}$ of $\R^d$ such that
\[
\R^d=S_r\oplus V_r, \ S_r=S_{r-1} \oplus V_{r-1}, \ldots, S_3=S_2\oplus V_2.
\]
Hence,
\[
\R^d=V_r\oplus V_{r-1} \oplus \ldots \oplus V_2 \oplus S_2.
\]
Define an operator $B$ on $\R^d$ by
\begin{align}\label{BZ-}
B\rvert_{S_2}=a_1\Id_{S_2}, \ B\rvert_{V_2}=a_2\Id_{V_2}, \ldots, B\rvert_{V_r}=a_r\Id_{V_r},
\end{align}
where $\Id_V$ denotes the identity map on a subspace $V\subset \R^d$. Moreover, 
set $B_n:=B$ and $\|\cdot \|_n=\| \cdot \|$ for $n<0$.

Set $\mathcal S':=\{ \|\cdot \|_n; \ n\in \Z\}$, $\mathbb B':=(B_n)_{n\in \Z}$, and let $\Sigma_{ED, \mathbb B', \mathcal S'}$ be defined as  in the  Definition \ref{def-SED}.
We now claim that $\Sigma_{ED, \mathbb B, \mathcal S}=\Sigma_{ED, \mathbb B', \mathcal S'}$. As mentioned above, $S_i$ does not depend on $c_i$. Thus, for any
$$
\tilde c_1\in (0, a_1), ~~\tilde c_i\in (b_{i-1}, a_{i}), ~~ i\in \{2,...,r\}, ~~  \tilde c_{r+1}\in (b_r, \infty),
$$
we have that the sequence $(\frac {1}{\tilde c_i}B_n)_{n\in \Z^+}$ also admits a strong exponential dichotomy with respect to the sequence of norms $\mathcal S$ and projections $\tilde P_n^i$ for all $n\in \Z^+$ such that
$\Ima \tilde P_0^i=S_{i}$. Moreover, it is easy to see from \eqref{BZ-} that the sequence $(\frac {1}{\tilde c_i}B_n)_{n\in \Z^-}$ admits a strong exponential dichotomy on $\Z^-$ with respect to the sequence of norms $\| \cdot \|_n$, $n\in \Z^-$ and  projections $\tilde P_n^i$, $n\in \Z^-$ such that $\Ker \tilde P_0^-=V_{i} \oplus \ldots \oplus V_r$. Since,
\[
\R^d=S_{i}\oplus V_{i} \oplus \ldots \oplus V_r,
\]
we conclude from the proof of~\cite[Theorem 2.3]{BDV0} that $(\frac {1}{\tilde c_i}B_n)_{n\in \Z}$ admits a strong exponential dichotomy on $\Z$ with respect to the sequence of norms $\mathcal S'$.
This implies that $(0, \infty)\setminus \Sigma_{ED, \mathbb B, \mathcal S} \subset (0, \infty)\setminus \Sigma_{ED, \mathbb B', \mathcal S'}$, i.e. $\Sigma_{ED, \mathbb B', \mathcal S'}\subset \Sigma_{ED, \mathbb B, \mathcal S}$. On the other hand, we trivially have that $\Sigma_{ED, \mathbb B, \mathcal S}\subset \Sigma_{ED, \mathbb B', \mathcal S'}$, proving that the two spectra coincide, i.e.,
\[
\Sigma_{ED, \mathbb B', \mathcal S'}
=\Sigma_{ED, \mathbb B, \mathcal S}
=\bigcup_{i=1}^r [a_i, b_i].
\]
The conclusion of the theorem now follows readily from Theorem~\ref{Ap} given in the Appendix A (which is a corollary of \cite[Theorem 2]{DZZ}).
\end{proof}

\subsection{Differentiable and H\"older linearization under exponential dichotomy}
Using similar arguments to the proof of Theorem \ref{t0}, we can also obtain the following result, where we remove the spectral gap condition, i.e., the first inequality of (\ref{sp2}), and obtain the simultaneously differentiable (at 0) and H\"older linearization.
\begin{theorem}\label{t1}
Suppose that $\mathbb B=(B_n)_{n\in \Z^+}$, $\mathcal S=\{\|\cdot \|_n; \ n\in \Z^+\}$ and $\Sigma_{ED, \mathbb B, \mathcal S}$  are as in the statement  of~Theorem {\rm \ref{t0}} and that \eqref{sp1}
and
\begin{equation}\label{sp3}
b_i/a_i<b_k^{-1}, \ \forall i=1, \ldots, k, \ b_j/a_j < a_{k+1}, \forall j=k+1, \ldots, r,
\end{equation}
hold. Let $f_n\colon \R^d \to \R^d$, $n\in \Z^+$, be a sequence of $C^1$ maps such that \eqref{fzero} and \eqref{fm} hold and let $\alpha_1\in \mathbb{R}$ be an arbitrary constant satisfying
\begin{equation}\label{alpha1}
    0<\alpha_1<\min\Big\{\frac{\ln a_{k+1}-\ln b_k}{\ln b_r},\frac{\ln a_{k+1}-\ln b_{k}}{\ln a_1^{-1}}\Big\}.
\end{equation}
Then, provided that $\eta$ is sufficiently small, there exists a sequence of homeomorphisms $h_n\colon \R^d\to \R^d$, $n\in \Z^+$, such that
\begin{itemize}
\item for $n\in \Z^+$,
$
h_{n+1} \circ (B_n+f_n)=B_n\circ h_n;
$
\item there exist constants $\tilde L, \varrho,\rho>0$ such that
\begin{align*}
\|h_n(x)-x\|_n=o(\|x\|_n^{1+\varrho}),\qquad
\|h_n^{-1}(x)-x\|_n=o(\|x\|_n^{1+\varrho})
\end{align*}
as $\|x\|_n\to 0$ and
\begin{align*}
\|h_n(x)-h_n(y)\|_n\!\le\! \tilde L\|x-y\|_n^{\alpha_1},~~~
\|h_n^{-1}(x)-h_n^{-1}(y)\|_n\!\le\! \tilde L \|x-y\|_n^{\alpha_1}
\end{align*}
for $x,y\in \mathbb{R}^d$ satisfying $\|x\|_n,\|y\|_n\le \rho$.
\end{itemize}
\end{theorem}

\begin{proof}
The proof can be obtained in a similar manner to the proof of Theorem~\ref{t0}. More precisely, we extend sequences $\mathbb B=(B_n)_{n\in \Z^+}$ and $(f_n)_{n\in \Z^+}$ to two-sided sequences $(B_n)_{n\in \Z}$ and $(f_n)_{n\in \Z}$
 exactly as in the proof of Theorem~\ref{t0}. Then, it remains to apply~\cite[Lemma 4]{DZZ20} to $B^*$ and $F$ introduced in Appendix A (see the proof of Theorem~\ref{Ap}).
\end{proof}

\section{Smooth linearization under nonuniform polynomial dichotomy}
\label{sec: lin pol}

In this section we present  main theorems concerning the linearization under polynomial behaviour. 

\subsection{$C^1$ linearization under polynomial dichotomy}
\begin{theorem}\label{theo: main linearization discrete}
Let $\mathbb A=(A_n)_{n\in \N}$ be a sequence of invertible operators on $\R^d$ that admits a nonuniform strong polynomial dichotomy and let $\mathcal S=\{{\|\cdot\|_n}; n\in \N\}$ be the sequence of norms given by Proposition~{\rm \ref{LN}}.
Suppose that $\Sigma_{PD, \mathbb A, \mathcal S}$ has the form~\eqref{PDA}, where the numbers $a_i$ and $b_i$ satisfy~\eqref{sp1}-\eqref{sp2}.
Moreover, let $g_n\colon \R^d \to \R^d$, $n\in \N$, be a sequence of $C^1$ maps  with the following properties:
\begin{itemize}
\item for $n\in \N$,
\begin{equation}\label{zero}
g_n(0)=0 \quad \text{and} \quad Dg_n(0)=0;
\end{equation}
\item there exists $c>0$ such that
\begin{equation}\label{non1}
\|Dg_n(x)\| \le \frac{c}{(n+1)^{1+2\varepsilon}}, \qquad \forall x\in \R^d;
\end{equation}
\item there exists $L>0$ such that
\begin{equation}\label{non2}
\| Dg_n(x)-Dg_n(y)\| \le \frac{L}{(n+1)^{1+2\varepsilon}}\|x-y\|, \quad \forall x,y\in \R^d,
\end{equation}
where $\varepsilon \ge 0$ is as in Definition~{\rm \ref{def-SPD}}.
\end{itemize}
Then, provided that $c$ is sufficiently small, there exists a sequence of $C^1$- diffeomorphisms $\psi_n:\mathbb{R}^d\to \mathbb{R}^d$, ${n\in \N}$, such that
\begin{equation}\label{E}
\psi_{n+1} \circ (A_n+g_n)=A_n\circ \psi_n, \quad n\in \N.
\end{equation}
Moreover, there exist $\tilde M, \tilde{\rho}>0$ such that 
\begin{equation}\label{EE}
\|D\psi_n(x)\|\le \tilde M n^{2\varepsilon} \quad \text{and} \quad \|D\psi_n^{-1}(x)\| \le \tilde M n^{2\varepsilon},
\end{equation}
for $n\in \N$ and $x\in\mathbb{R}^d$ satisfying $\|x\|\le \frac{\tilde \rho}{n^{2\varepsilon}}$.
\end{theorem}

\begin{proof}
We begin by recalling that~\eqref{ln1} holds with $\delta:=2\varepsilon \ge 0$  (see the proof of Proposition~\ref{LN}).
Setting $G_n:=A_n+g_n$ for $n\in \N$ and
\begin{align*}
\mathcal G(m,n):=
\begin{cases}
G_{m-1} \circ \ldots \circ G_n, & m>n;
\\
\Id, & m=n,
\end{cases}
\end{align*}
we easily verify that
\begin{equation}\label{G}
\mathcal G(m,n)(x)=\cA(m,n)x+\sum_{j=n}^{m-1}\cA(m, j+1)g_j ( \mathcal G(j,n)(x))
\end{equation}
for $m\ge n$ and $x\in \R^d$, and therefore
\begin{equation}\label{DG}
D\mathcal G(m,n)(x)=\cA(m,n)+\sum_{j=n}^{m-1} \cA(m,j+1) Dg_j(\mathcal G(j,n)(x)) D\mathcal G(j,n)(x).
\end{equation}
Observe that~\eqref{ln1} and~\eqref{non1} imply that 
\begin{equation}\label{derofg}
\|Dg_n(x)v\|_{n+1} \le C(n+1)^{2\varepsilon}\|Dg_n(x)v\|\le \frac{cC}{n+1}\|v\|\le \frac{cC}{n+1}\|v\|_n
\end{equation}
and therefore
\[
\begin{split}
&\|D\mathcal G(m,n)(x) v\|_m  \\
&\le K\bigg (\frac m n \bigg )^a \|v\|_n+K \sum_{j=n}^{m-1} \bigg (\frac{m}{j+1} \bigg )^a\| Dg_j(\mathcal G(j,n)(x)) D\mathcal G(j,n)(x)v\|_{j+1} \\
&\le K\bigg (\frac m n \bigg )^a \|v\|_n+K \sum_{j=n}^{m-1} \bigg (\frac{m}{j+1} \bigg )^a \frac{cC}{j+1} \|D\mathcal G(j,n)(x)v\|_j
\end{split}
\]
by \eqref{bg1sn},
which is equivalent to
\[
\bigg (\frac n m\bigg )^a \|D\mathcal G(m,n)(x) v\|_m \le K\|v\|_n+K\sum_{j=n}^{m-1} \frac{cC}{j+1}\bigg (\frac n j\bigg )^a\|(D\mathcal G(j,n))(x)v\|_j
\]
for $m\ge n$ and $x, v\in \R^d$.  We need the following discrete-version of Gronwall's lemma (see e.g. \cite[Lemma 4.32]{E-book}).
\begin{lemma}\label{lm-G}
Let $n\in \N$ and $\alpha >0$. Suppose that $(u_m)_{m\ge n}$ and $(z_m)_{m\ge n}$ are two nonnegative sequences satisfying 
$
u_m \le K\{u_0 + \sum_{j=n}^{m-1} z_ju_j\}
$
for $m \ge n$.
Then
\begin{align*}
u_m \le Ku_0\,e^{\sum_{j=n}^{m-1}(Kz_j)}, \quad \forall m \ge n.
\end{align*}
\end{lemma}
\noindent By the above lemma, we find that
\[
\bigg (\frac n m\bigg )^a\| D\mathcal G(m,n)(x)v\|_m\le K\|v\|_n e^{\sum_{j=n}^{m-1}\frac{cC K}{j+1}}.
\]
Since ${\sum_{j=n}^{m-1}\frac{1}{j+1}}\le \log m-\log n$,
we obtain that
\begin{equation}\label{204}
\| D\mathcal G(m,n)(x)v\|_m \le K \bigg (\frac mn\bigg )^{a+cC K}\|v\|_n,
\end{equation}
for $m\ge n\ge 1$ and $x, v\in \R^d$.

Furthermore, we see from (\ref{DG}) that
\begin{align}\label{zz}
&D\mathcal G(m,n)(x)-D\mathcal G(m, n)(y) 
\nonumber\\
&=\phantom{=}\sum_{j=n}^{m-1} \cA(m,j+1)Dg_j(\mathcal G(j,n)(x))D\mathcal G(j,n)(x) 
\nonumber\\
&\phantom{=}-\sum_{j=n}^{m-1} \cA(m,j+1)Dg_j(\mathcal G(j,n)(y))D\mathcal G(j,n)(y) 
\nonumber\\
&=\phantom{=}\sum_{j=n}^{m-1}\cA(m,j+1)Dg_j(\mathcal G(j,n)(x))(D\mathcal G(j,n)(x)-D\mathcal G(j,n)(y))
\nonumber\\
&\phantom{=}-\sum_{j=n}^{m-1}\cA(m,j+1)(Dg_j(\mathcal G(j,n)(y))-Dg_j(\mathcal G(j,n)(x)))D\mathcal G(j,n)(y).
\end{align}
In addition, \eqref{bg1sn}, \eqref{ln1} and~\eqref{non1} imply that
\begin{align}\label{zz1}
&\bigg\|\sum_{j=n}^{m-1}\cA(m,j+1)Dg_j(\mathcal G(j,n)(x))(D\mathcal G(j,n)(x)-D\mathcal G(j,n)(y))v\bigg\|_m
\nonumber\\
&\le K\sum_{j=n}^{m-1}\bigg (\frac{m}{j+1}\bigg )^a \frac{cC}{j+1}\|(D\mathcal G(j,n)(x)-D\mathcal G(j,n)(y))v\|_j.
\end{align}
Furthermore, by~\eqref{bg1sn}, \eqref{ln1}, \eqref{non2} and~\eqref{204} we have that
\begin{align}\label{zz2}
&\bigg\|\sum_{j=n}^{m-1}\cA(m,j+1)(Dg_j(\mathcal G(j,n)(y))-Dg_j(\mathcal G(j,n)(x)))D\mathcal G(j,n)(y)v\bigg\|_m
\nonumber\\
&\le \sum_{j=n}^{m-1} K\bigg (\frac{m}{j+1}\bigg )^a \|
(Dg_j(\mathcal G(j,n)(y))-Dg_j(\mathcal G(j,n)(x)))D\mathcal G(j,n)(y)v\|_{j+1}
\nonumber\\
&\le \sum_{j=n}^{m-1} KC\bigg (\frac{m}{j+1}\bigg )^a (j+1)^{2\varepsilon}  \|
(Dg_j(\mathcal G(j,n)(y))-Dg_j(\mathcal G(j,n)(x)))D\mathcal G(j,n)(y)v\|
\nonumber\\
&\le \sum_{j=n}^{m-1} \bigg (\frac{m}{j+1}\bigg )^a (j+1)^{2\varepsilon}\frac{KLC}{(j+1)^{1+2\varepsilon}}
\|\mathcal G(j,n)(x)-\mathcal G(j,n)(y)\| \cdot \|D\mathcal G(j,n)(y)v\| 
\nonumber\\
&\le \sum_{j=n}^{m-1} \bigg (\frac{m}{j+1}\bigg )^a  \frac{KLC}{j+1}\|\mathcal G(j,n)(x)-\mathcal G(j,n)(y)\|_j \cdot \|D\mathcal G(j,n)(y)v\|_j 
\nonumber\\
&\le \sum_{j=n}^{m-1} \bigg (\frac{m}{j+1}\bigg )^a  \frac{KLC}{j+1}K^2 \bigg (\frac j n\bigg )^{2a+2cCK}\|x-y\|_n \cdot \|v\|_n 
\nonumber\\
&\le K^3LC\bigg (\frac m n\bigg )^{2a+2cCK}~\sum_{j=n}^{m-1}\frac{1}{j+1}\,\|x-y\|_n\cdot \|v\|_n
\nonumber\\
&\le K^3LC\bigg (\frac m n\bigg )^{2a+2cCK+1}\|x-y\|_n\cdot \|v\|_n.
\end{align}
 Hence,
from~\eqref{zz}, \eqref{zz1} and~\eqref{zz2}, we obtain that
\[
\begin{split}
&\|(D\mathcal G(m,n)(x)-D\mathcal G(m, n)(y))v\|_m \\
 &\le K \sum_{j=n}^{m-1}\bigg (\frac{m}{j+1}\bigg )^a \frac{cC}{j+1}\|(D\mathcal G(j,n)(x)-D\mathcal G(j,n)(y))v\|_j \\
&\phantom{\le}+K^3LC\bigg (\frac m n\bigg )^{2a+2cCK+1}\|x-y\|_n \cdot \|v\|_n.
\end{split}
\]
This together with Lemma~\ref{lm-G}  implies that
\begin{equation}\label{175}
\|(D\mathcal G(m,n)(x)-D\mathcal G(m, n)(y))v\|_m\le
\tilde K\bigg (\frac m n\bigg )^{\tilde a}\|x-y\|_n \cdot \|v\|_n,
\end{equation}
for $m\ge n\ge 1$ and $x, y, v\in \R^d$, where $\tilde K:= K^3LC>0$ and $\tilde a:=2a+3cCK+1>0$.

On the other hand, observe that~\eqref{G} implies that
\[
\mathcal G(2^{m}, 2^{n})(x)=\mathcal B(m, n)x+\sum_{j=2^{n}}^{2^{m}-1} \cA(2^{m}, j+1)g_j(\mathcal G(j, 2^{n})(x)),
\]
for $m\ge n\ge 0$ and $x\in \R^d$. For $n\in \Z^+$ and $x\in \R^d$, set
\begin{align*}
f_n(x)
&:=\mathcal G(2^{n+1}, 2^{n})(x)-\mathcal B(n+1, n)x
\\
&=\sum_{j=2^n}^{2^{n+1}-1}\cA(2^{n+1}, j+1)g_j(\mathcal G(j, 2^n)(x)).
\end{align*}
Then,
\[
Df_n(x)=\sum_{j=2^n}^{2^{n+1}-1}\cA(2^{n+1}, j+1) Dg_j(\mathcal G(j, 2^n)(x))D\mathcal G(j, 2^n)(x).
\]
It is clear that $f_n$ for $n\in\Z^+$ satisfy~\eqref{fzero} by~\eqref{zero}.

In order to show that $f_n$ satisfy~\eqref{fm} with the sequence of norms $\{\|\cdot\|_{2^n}; n\in \Z^+\}$, we note that~\eqref{bg1sn}, \eqref{derofg} and \eqref{204}  imply that
\[
\begin{split}
\|Df_n(x)v \|_{2^{n+1}} &\le \sum_{j=2^n}^{2^{n+1}-1}K\bigg (\frac{2^{n+1}}{j+1} \bigg )^a \frac{cC}{j+1}K \bigg (\frac{j}{2^n} \bigg )^{a+cCK}\|v\|_{2^n}
\\
&\le K^2cC\,2^{2a+cCK}  \sum_{j=2^n}^{2^{n+1}-1}\frac{1}{j+1}\|v\|_{2^n}
\le \eta \|v\|_{2^n},
\end{split}
\]
for $n\in\Z^+$ and $x,v\in \R^d$ , where $\eta:=(K^2C\,2^{2a+cCK+1})c>0$.
This proves the first inequality of~\eqref{fm}.
Note that we can make $\eta$ sufficiently small provided that we take $c$ small enough. Moreover, \eqref{175} gives that
\[
\begin{split}
\|(Df_n(x)-Df_n (y))v\|_{2^{n+1}} &=\|(D\mathcal G(2^{n+1}, 2^n)(x)-D\mathcal G(2^{n+1}, 2^n)(y))v\|_{2^{n+1}}\\
 & \le \tilde K2^{\tilde a}\|x-y\|_{2^n} \cdot \|v\|_{2^n},
\end{split}
\]
for $n\in \Z^+$ and $x, y, v\in \R^d$.
This proves the second inequality of~\eqref{fm} and therefore \eqref{fm} is proved.

Applying Theorem~\ref{t0} to $\mathbb B=(B_n)_{n\in \Z^+}$ given by~\eqref{BN} with the sequence of norms $\{\| \cdot \|_{2^n}; \ n\in \Z^+\}$, we conclude that there exists a sequence of $C^1$-diffeomorphisms $h_n\colon \R^d \to \R^d$, $n\in \Z^+$, such that
\begin{equation}\label{LIN}
h_{n+1}\circ (B_n+f_n)=B_n\circ h_n, \quad n\in \Z^+.
\end{equation}
Moreover, there exist $M, \rho>0$ such that
\begin{equation}\label{LIN2}
\|D h_n(x)v\|_{2^n} \le M\|v\|_{2^n}, \quad \|Dh_n^{-1}(x)v\|_{2^n} \le M\|v\|_{2^n}, \quad \forall n\in \Z^+,
\end{equation}
for $v\in \R^d$ and $x\in \R^d$  satisfying $\|x\|_{2^n}\le \rho$.
Hence,
\begin{equation}\label{h}
h_{n+1}\circ \mathcal G(2^{n+1}, 2^n)=\cA(2^{n+1}, 2^n)\circ h_n, \quad n\in \Z^+.
\end{equation}
Letting $k\in \N$ and $n\in \Z^+$ be such that $2^n\le k<2^{n+1}$, we set
\begin{equation}\label{psik}
\psi_k:=\cA(k, 2^n)\circ h_n\circ \mathcal G(2^n, k)
\end{equation}
and assert that~\eqref{E} holds. For this purpose, we first assume that $2^n\le k<k+1<2^{n+1}$. Then,
\[
\begin{split}
\psi_{k+1}\circ (A_k+g_k)&=\cA(k+1, 2^n)\circ h_n \circ \mathcal G(2^n, k+1) \circ (A_k+g_k) \\
&=A_k\circ \cA(k, 2^n)\circ h_n \circ \mathcal G(2^n, k) \\
&=A_k\circ \psi_k,
\end{split}
\]
and thus~\eqref{E} holds. Let us now suppose that $k+1=2^{n+1}$. Then,  using~\eqref{h} we have that
\[
\begin{split}
A_k\circ \psi_k 
&=\cA(k+1, 2^n)\circ h_n \circ \mathcal G(2^n, k) \\
&=\cA(2^{n+1},2^n)\circ h_n \circ \mathcal G(2^n, k)\\
&=h_{n+1} \circ \mathcal G(2^{n+1}, 2^n)\circ \mathcal G(2^n, k)\\
&=h_{n+1} \circ \mathcal G(k+1, k) =h_{n+1} \circ (A_k+g_k)=\psi_{k+1} \circ (A_k+g_k),
\end{split}
\]
where the last equality holds because $k+1=2^{n+1}$ implies that
\[
\psi_{k+1}=\mathcal G(k+1, 2^{n+1})\circ h_{n+1}\circ \cA(2^{n+1}, k+1)=h_{n+1}.
\]
Consequently, \eqref{E} again holds. Finally, using~\eqref{bg1sn}, \eqref{ln1}, \eqref{204} and~\eqref{LIN2} we conclude that
\[
\begin{split}
\|D\psi_k(x)v\| &\le \|D\psi_k(x)v\|_{k} 
\le K2^{a+cCK} 
\|Dh_n(\cA(2^n, k)x)
\cA(2^n, k)v\|_{2^n} 
\\
&\le K^2M2^{2a+cCK} \|v\|_k \le CK^2M2^{2a+cCK}k^{2\varepsilon} \|v\|
\end{split}
\]
whenever $\|x\|\le \frac{\rho}{CK2^ak^{2\varepsilon}}$ since it implies
$$
\|\cA(2^n, k)x\|_{2^n}\le K2^a \|x\|_k \le CK2^a k^{2\varepsilon}\|x\| \le \rho.
$$
Hence, the first estimate in~\eqref{EE} holds. Similarly, one can establish the second estimate in~\eqref{EE}. The proof of this theorem is completed.
\end{proof}

  We now apply Theorem \ref{theo: main linearization discrete} to a very simple example, which is not covered by previously known results.
\begin{example}\label{example: main theo discrete}
Let us consider the sequences of linear operators $\mathbb{A}=(A_n)_{n\in \N}$ and $(P_n)_{n\in \N}$ acting on $\R^2$ given by
\[
A_n:=\begin{pmatrix}
\frac{n}{n+1} & 0\\
0 & \frac{n+1}{n}
\end{pmatrix} \quad\text{ and }\quad P_n:=\begin{pmatrix}
1 & 0\\
0 & 0
\end{pmatrix},\quad n\in \N.
\]
Then,
\[
\cA(m,n)=\begin{pmatrix}
\frac{n}{m} & 0\\
0 & \frac{m}{n}
\end{pmatrix} \quad  \text{for every } m,n\in \N.
\]
It is easy to see that $(A_n)_{n\in\N}$ admits a nonuniform strong polynomial dichotomy with constants $K=a=\lambda=1$, $\varepsilon=0$. Furthermore, it admits a strong polynomial dichotomy with respect to $\mathcal{S}=\{\|\cdot\|_n; \ n\in \N\}$, where $\|\cdot \|_n=\| \cdot \|$ is the Euclidean norm. It is easy to see that $\Sigma_{PD,\mathbb{A},\mathcal{S}}=\{-1,1\}$. In particular, $r=2$, $a_1=b_1=\frac{1}{2}$ and $a_2=b_2=2$. Thus, conditions \eqref{sp1} and \eqref{sp2} are satisfied. 

Consider now $\xi:\R\to \R$ given by $\xi(x)=x^2e^{-x^2}$. Then, $D\xi(x)=2xe^{-x^2}(1-x^2)$ and, consequently, $|D\xi(x)|\leq 1$ and $|D\xi(x)-D\xi(y)|\leq 2|x-y|$. Thus, taking $g_n:\R^2\to \R^2$ as
\[g_n(x_1,x_2)=\frac{c}{n+1}(\xi(x_1),\xi(x_2)), \quad n\in \N,\]
where $c>0$ is a constant, it follows that conditions \eqref{zero}, \eqref{non1} and \eqref{non2} are satisfied. In particular, Theorem \ref{theo: main linearization discrete} may be applied whenever $c>0$ is small enough. On the other hand, it is easy to see that the sequence  $\mathbb{A}=(A_n)_{n\in \N}$ does not admit a nonuniform strong  exponential dichotomy. Therefore, previously known results such as the one in \cite{DZZ} can not be applied to this example.
\end{example}

\subsection{Differentiable and H\"older linearization  under polynomial dichotomy}
In the next result we remove the spectral gap condition from our hypothesis (i.e., the first inequality of \eqref{sp2}), and obtain a linearization which is simultaneously differentiable (at 0) and H\"older continuous in a neighborhood of 0.

\begin{theorem}\label{theo: diff+holder linearization}
Suppose that $\mathbb A=(A_n)_{n\in \N}$, $\mathcal S=\{ \| \cdot \|_n; \ n\in \N\}$ and $\Sigma_{PD, \mathbb A, \mathbb S}$ are given as  in Theorem~\ref{theo: main linearization discrete} and that \eqref{sp1} and \eqref{sp3} hold.
Let $g_n:\mathbb{R}^d\to \mathbb{R}^d$, $n\in \N$, be a sequence of $C^1$ maps $g_n\colon \R^d \to \R^d$ such that \eqref{zero}-\eqref{non2} hold, and let $\alpha_1$ and $\varrho$ be given in Theorem {\rm \ref{theo: main linearization discrete}}.
Then, provided that $c$ is sufficiently small, there exists a sequence of homeomorphisms $\psi_n:\mathbb{R}^d\to \mathbb{R}^d$, $n\in \N$, such that \eqref{E} holds.
Moreover, there exist constants $ L', \rho'>0$ such that 
\begin{equation}\label{000}
\psi_n(x)=x+n^{2\varepsilon(1+\varrho)} o(\|x\|^{1+\varrho}),\quad
\psi_n^{-1}(x)=x+n^{2\varepsilon(1+\varrho)} o(\|x\|^{1+\varrho})
\end{equation}
as $\|x\|\to 0$ and
\begin{equation}\label{001}
\|\psi_n(x)-\psi_n(y)\|\!\le\!  L'n^{2\varepsilon \alpha_1 }\|x-y\|^{\alpha_1},\quad
\|\psi_n^{-1}(x)-\psi_n^{-1}(y)\|\!\le\! \ L' n^{2\varepsilon \alpha_1} \|x-y\|^{\alpha_1}
\end{equation}
for $x,y\in \mathbb{R}^d$ satisfying $\|x\|,\|y\|\le \frac{\rho'}{n^{2\varepsilon}}$, where $\varepsilon$ is as in the Definition~\ref{def-SPD}.
\end{theorem}

\begin{proof}
We use the same notation as in the proof of Theorem~\ref{theo: main linearization discrete}. By Theorem~\ref{t1}, there exist a sequence of homeomorphisms $h_n\colon \R^d \to \R^d$, $n\in \Z^+$ satisfying~\eqref{LIN} and  constants $\tilde L, \varrho,\rho>0$  such that 
\begin{equation}\label{cd0}
\|h_n(x)-x\|_{2^n}=o(\|x\|_{2^n}^{1+\varrho}),\quad
\|h_n^{-1}(x)-x\|_{2^n}=o(\|x\|_{2^n}^{1+\varrho})
\end{equation}
as $\|x\|_{2^n}\to 0$ and
\begin{equation}\label{cd1}
\|h_n(x)-h_n(y)\|_{2^n}\!\le\! \tilde L\|x-y\|_{2^n}^{\alpha_1},\quad
\|h_n^{-1}(x)-h_n^{-1}(y)\|_{2^n}\!\le\! \tilde L \|x-y\|_{2^n}^{\alpha_1},
\end{equation}
for $x,y\in\mathbb{R}^d$ satisfying $\|x\|_{2^n},\|y\|_{2^n}\le \rho$. 

Next, we may construct a sequence of homeomorphisms $\psi_k \colon \R^d \to \R^d$, $k\in \N$, exactly as in the proof of Theorem~\ref{theo: main linearization discrete} (see~\eqref{psik}) such that~\eqref{E} holds.  Moreover, using~\eqref{bg1sn}, \eqref{ln1}, \eqref{204}, \eqref{cd1} we have that
\[
\begin{split}
\| \psi_k(x)-\psi_k(y)\| &=\|\cA(k, 2^n)h_n({\mathcal G}(2^n, k)(x))-\cA(k, 2^n)h_n({\mathcal G}(2^n, k)(y))\| \\
&\le \|\cA(k, 2^n)h_n({\mathcal G}(2^n, k)(x))-\cA(k, 2^n)h_n({\mathcal G}(2^n, k)(y))\|_k\\
&\le K2^a \|h_n({\mathcal G}(2^n, k)(x))-h_n({\mathcal G}(2^n, k)(y))\|_{2^n} \\
&\le K\tilde L2^a\|{\mathcal G}(2^n, k)(x)-{\mathcal G}(2^n, k)(y)\|_{2^n}^{\alpha_1} \\
&\le K^{1+\alpha_1}\tilde L2^{a(1+\alpha_1)+cCK\alpha_1} \|x-y\|_{k}^{\alpha_1}\\
&\le C^{\alpha_1}K^{1+\alpha_1}\tilde L2^{a(1+\alpha_1)+cCK\alpha_1}k^{2\varepsilon \alpha_1} \|x-y\|^{\alpha_1} \\
&\le L'k^{2\varepsilon \alpha_1} \|x-y\|^{\alpha_1},
\end{split}
\]
for $2^n\le k<2^{n+1}$ and $x, y\in \R^d$ such that $\|x\|, \|y\| \le \frac{\rho}{CK2^a k^{2\varepsilon}}$, 
where $L':=C^{\alpha_1}K^{1+\alpha_1}\tilde L2^{a(1+\alpha_1)+cCK\alpha_1}>0$. Hence, the first estimate in~\eqref{001} is proved by putting $\rho':=\frac{\rho}{CK2^a}$. Similarly, one can establish the second one.

We now claim that for each $k\in \N$, we have that 
\begin{equation}\label{cd0-k}
    \|\psi_k(x)-x\|_k=o(\|x\|_k^{1+\varrho})\quad \mbox{as $\|x\|_k\to 0$}.
\end{equation}
When $k=2^n$, \eqref{cd0-k} follows readily from~\eqref{cd0} since $\psi_k=h_n$. Assume now that~\eqref{cd0-k} 
holds for some $2^n\le k<2^{n+1}-1$. By~\eqref{psik}, we have that 
$
\psi_{k+1}=A_k\circ \psi_k \circ (A_k+g_k)^{-1}.
$
Hence,  
\[
\begin{split}
\psi_{k+1}-\Id &=A_k\circ \psi_k \circ (A_k+g_k)^{-1}-\Id \\
&=A_k\circ (\Id+{\psi}_k-\Id)\circ (A_k^{-1}+(A_k+g_k)^{-1}-A_k^{-1})-\Id \\
&=A_k \circ ((A_k+g_k)^{-1}-A_k^{-1})+A_k\circ ({\psi}_k-\Id)\circ (A_k+g_k)^{-1}.
\end{split}
\]
In order estimate $\psi_{k+1}-\Id$, we see from~\eqref{ln1}, \eqref{non1} and~\eqref{non2} (together with the mean-value theorem) that 
\begin{align*}
&\|A_k \circ ((A_k+g_k)^{-1}-A_k^{-1})(x)\|_{k+1}
\\
&= \|g_k\circ(A_k+g_k)^{-1}(x)\|_{k+1}
\le C(k+1)^{2\varepsilon}\|g_k\circ(A_k+g_k)^{-1}(x)\|
\\
&\le C(k+1)^{2\varepsilon}\frac{cL}{(k+1)^{2+4\varepsilon}}
\|(A_k+g_k)^{-1}(x)\|^2
\\
&\le cCL\|(A_k+g_k)^{-1}(x)\|_k^2=o(\|x\|_{k+1})
\end{align*}
as $\|x\|_{k+1}\to 0$, which implies that $\|(A_k+g_k)^{-1}(x)\|_{k}\to 0$, as known from~\eqref{204}. Moreover, 
we see from that \eqref{bg1sn} and~\eqref{cd0-k} that
\[
\begin{split}
&\|A_k\circ ({\psi}_k-\Id)\circ (A_k+g_k)^{-1}(x)\|_{k+1}
\\
&\le  K2^a\|({\psi}_k-\Id)\circ (A_k+g_k)^{-1}(x)\|_k 
\\
&= o(\|(A_k+g_k)^{-1}(x)\|_k)
=o(\|x\|_{k+1})
\end{split}
\]
as $\|x\|_{k+1}\to 0$.
From the last two estimates we conclude that~\eqref{cd0-k} holds for $k+1$. By induction, we obtain that~\eqref{cd0-k} holds for each $k\in \N$.
Thus, recalling~\eqref{ln1} we have that 
\begin{equation*}
\|\psi_k(x)-x\|\le\|\psi_k(x)-x\|_k=o(\|x\|_{k}^{1+\varrho})=k^{2\varepsilon(1+\varrho)}o(\|x\|^{1+\varrho})
\end{equation*}
as $\|x\|\to 0$, which proves the first estimate of~\eqref{000}. The second one of~\eqref{000} can be proved similarly and the proof is completed. 
\end{proof}

\subsection{The case of infinite-dimensional dynamics}\label{sec: inf dim discrete}
The purpose of this short subsection is to indicate how one can extend our results to the infinite-dimensional setting. Let $X$ be an arbitrary Hilbert space and $\mathcal B(X)$ be a space of all bounded linear operators on $X$.  

One can introduce all relevant notions of dichotomy (as in Definitions~\ref{def-SPD}-\ref{1753}) for sequences of invertible operators in $\mathcal B(X)$. Moreover, one can establish versions of Propositions~\ref{LN}, \ref{LN2} and~\ref{WW} in the infinite-dimensional setting by arguing exactly as we did in the finite-dimensional setting. Indeed, note that in the proofs of these results finite dimensionality was never used.

Moreover, for a sequence $\mathbb A$ of invertible operators in $\mathcal B(X)$ and a sequence of norms $\mathcal S$ on $X$, one can introduce $\Sigma_{PD, \mathbb A, \mathcal S}$ and $\Sigma_{ED, \mathbb A, \mathcal S}$ as in Definitions~\ref{def-01} and~\ref{def-SED}. Then, Corollary~\ref{KOR} remains valid. However, Proposition~\ref{22} (and consequently also Corollary~\ref{KOR1}) fails to hold in the infinite-dimensional setting. Hence, one can formulate versions of Theorems~\ref{theo: main linearization discrete}
and \ref{theo: diff+holder linearization} in the infinite-dimensional case by assuming an additional condition that $\Sigma_{PD, \mathbb A, \mathcal S}$ has the form~\eqref{PDA} (which in the finite-dimensional setting is automatically satisfied).

In particular, we note that the preparatory results Theorem~\ref{t0} and Theorem~\ref{Ap} can be established in exactly the same manner. We only emphasize that the assumption that $X$ is a Hilbert space would be used in the construction of (closed) subspaces $S_i$ appearing in the proof of Theorem~\ref{t0}.

\section{The case of continuous time}\label{sec: continuous time case}
In this section we present versions of Theorems \ref{theo: main linearization discrete} and \ref{theo: diff+holder linearization} in the case of continuous time dynamical systems. 

Let us consider a linear nonautonomous equation 
\begin{equation}\label{LDE}
x'=A(t)x \quad t\ge 1,
\end{equation}
where $A:[1,\infty)\to \R^{d\times d}$ is a continuous map acting on $[1, \infty)$ and with values in the family of linear operators on $\R^d$. By $T(t,s)$ we will denote the evolution family associated with~\eqref{LDE}.
 Moreover, given a continuous function  $f\colon [1, \infty) \times \R^d \to \R^d$, we consider the following semilinear differential equation
\begin{equation}\label{NDE}
x'=A(t)x+f(t,x), \quad t\ge 1.
\end{equation}

\subsection{Polynomial dichotomies for continuous time dynamics}
The main notions of dichotomy we are going to use in this section are the following.

\begin{definition}\label{def: nonunif pol dich continuous}
We say that~\eqref{LDE} admits a \emph{nonuniform strong polynomial dichotomy} if there exist $K>0$, $a\geq \lambda>0$, $\varepsilon\geq 0$ and a family of projections $P(t)$, $t\ge 1$ on $\R^d$ such that the following properties hold:
\begin{itemize}
\item  for $t\ge s\ge 1$,
\begin{equation}\label{Pro}
P(t)T(t,s)=T(t,s)P(s);
\end{equation}
\item for $t\ge  s$,
\begin{equation}\label{B1}
\| T(t,s)P(s)\| \le K \bigg (\frac t s\bigg )^{-\lambda}s^\varepsilon,
\end{equation}
and 
\begin{equation}\label{B2}
\|T(s,t)Q(t)\| \le K \bigg (\frac t s\bigg )^{-\lambda}t^\varepsilon,
\end{equation}
where $Q(t)=\Id-P(t)$;
\item for $t\ge s$,
\begin{equation}\label{bnd}
\| T(t,s)\| \le K\bigg (\frac t s \bigg )^as^\varepsilon \quad \text{and} \quad \| T(s,t)\| \le K \bigg(\frac t s\bigg )^at^\varepsilon.
\end{equation}
\end{itemize}
In particular, when $\varepsilon=0$ we say that~\eqref{LDE} admits a uniform strong polynomial
dichotomy.
\end{definition}

\begin{definition}
Let $\mathcal{S}=\{\| \cdot \|_t; \ t\ge 1 \}$ be a family of norms on $\R^d$.
We say that \eqref{LDE} admits a \emph{strong polynomial dichotomy with respect to the family of norms $\mathcal{S}$} if there exist $K>0$, $a\geq \lambda>0$ and a family of projections $P(t)$, $t\ge 1$ on $\R^d$ such that the following properties hold:
\begin{itemize}
\item \eqref{Pro} is satisfied for every $t\geq s\geq 1$;
\item for $t\geq s\geq 1$ and every $x\in \mathbb{R}^d$, 
\begin{equation}\label{d1}
\| T(t,s)P(s) x\|_t \leq K\left(\frac{t}{s}\right)^{-\lambda} \| x\|_s
\end{equation}
and 
\begin{equation}\label{d2}
\| T(s,t)Q(t) x \|_s \leq  K \left(\frac{t}{s}\right)^{-\lambda}\| x\|_t,
\end{equation}
where $Q(t)=\Id-P(t)$;
\item for $t\ge s \geq 1$ and $x\in \mathbb{R}^d$,
\begin{equation}\label{bnd-norm}
\| T(t,s)x\|_t \le K\bigg (\frac t s \bigg )^a\|x\|_s \quad \text{and} \quad \| T(s,t)x\|_s \le K \bigg(\frac t s\bigg )^a\|x\|_t.
\end{equation}
\end{itemize}
\end{definition}

As in the discrete time case, it turns out that the previous notions are strongly related.
\begin{proposition}\label{prop: equiv nonunifXnorms cont}
The following properties are equivalent:
\begin{enumerate}
\item \eqref{LDE} admits a nonuniform strong polynomial dichotomy;
\item \eqref{LDE} admits a strong polynomial dichotomy with respect to a family of norms $\mathcal{S}=\{\| \cdot \|_t; \ t\ge 1\}$ with the property that there exist $C>0$ and $\delta \ge 0$ such that 
\begin{equation}\label{ln}
\| x\| \le \| x\|_t \le Ct^\delta \| x\|, \quad x\in \mathbb{R}^d \text{ and }  t\ge 1.
\end{equation}
\end{enumerate}
\end{proposition}

\begin{proof} 
Assume \eqref{LDE} admits a nonuniform strong polynomial dichotomy. For $t \ge 1$ and $x\in \R^d$, set 
\[\|x\|_t:=\|x\|^s_t+\|x\|^u_t
\]
where
\[
\begin{split}
\|x\|^s_t &=\sup_{r\ge t} \left(\| T(r, t)P(t) x\|\left(\frac{r}{t}\right)^\lambda \right) +\sup_{r\leq t}\left( \| T(r, t)P(t) x\rVert \left(\frac{t}{r}\right)^{-a}\right)
\end{split}
\]
and
\[
\begin{split}
\|x\|^u_t &=\sup_{r\leq t} \left(\| T(r, t)Q(t)x\|\left(\frac{t}{r}\right)^\lambda \right) +\sup_{r\geq t}\left( \| T(r, t)Q(t) x\rVert \left(\frac{r}{t}\right)^{-a}\right).
\end{split}
\]
Then, proceeding as in the proof of Proposition \ref{LN} we conclude that (2) is satisfied. Similarly, if (2) is satisfied then arguing as in the proof of Proposition \ref{LN} we conclude that (1) is satisfied.
\end{proof}

In the next proposition we present a relationship between the dynamics of \eqref{LDE} and a suitable discretization of it. 
\begin{proposition} \label{prop: equiv pol dich contXdisc}
Let $\mathcal{S}=\{\| \cdot \|_t; \ t\ge 1\}$ be a family of norms on $\R^d$.
Assume that there exist $K,a>0$ such that the evolution family $T(t,s)$ of \eqref{LDE} satisfies \eqref{bnd-norm}.
Then, the following properties are equivalent:
\begin{itemize}
\item \eqref{LDE} admits a strong polynomial dichotomy with respect to the family of norms $\mathcal{S}$;
\item the sequence $(A_n)_{n\in \N}$ admits a strong polynomial dichotomy with respect to the family of norms $\{\|\cdot\|_n; \ n\in \N\}$, where
\begin{equation}\label{AN}
A_n=T(n+1, n), \quad n\in \N.
\end{equation}
\end{itemize}
\end{proposition}

\begin{proof}
Considering $\cA(m,n)$ as in \eqref{eq: cocycle} with $(A_n)_{n\in \N}$ given by \eqref{AN} we have that 
\begin{equation}\label{AT}
\cA(m,n)=T(m,n), \quad m, n\in \N.
\end{equation}
Assume that~\eqref{LDE} admits a strong polynomial dichotomy with respect to the family of norms $\mathcal{S}$ and projections $P(t)$, $t\ge 1$. By~\eqref{d1} and~\eqref{AT}, we have that 
\[
\| \cA(m,n)P(n)x\|_m \le K \bigg (\frac m n\bigg )^{-\lambda}\|x\|_n, \quad m\ge n \text{ and }x\in\R^d.
\]
Similarly, \eqref{d2} and~\eqref{AT} imply that 
\[
\| \cA(n,m)Q(m)\|_n \le K \bigg (\frac m n\bigg )^{-\lambda}\|x\|_m, \quad m\ge n \text{ and }x\in\R^d.
\]
Moreover, by \eqref{bnd-norm} and~\eqref{AT} we get that for every $x\in \R^d$ and $m\ge n$,
\begin{equation*}
\| \cA(m,n)x\|_m \le {K}\bigg (\frac m n\bigg )^{a}\|x\|_n \quad \text{and} \quad \| \cA(n,m)x\|_n \le {K}\bigg (\frac m n\bigg )^{a}\|x\|_m.
\end{equation*}
Consequently, we conclude that the sequence $(A_n)_{n\in \N}$ admits a nonuniform strong polynomial dichotomy with respect to the sequence of norms $\{\|\cdot\|_n; \ n\in \N\}$ and projections $P(n)$, $n\in \N$.

Assume now that the sequence $(A_n)_{n\in \N}$ admits a strong polynomial dichotomy with respect the sequence of norms $\{ \| \cdot \|_n; \  n\in \N\}$ and projections $P_n$, $n\in \N$. Hence, there exist $K>0$ and $a\geq \lambda >0$ such that~\eqref{pro}, \eqref{pd1sn} and~\eqref{bg1sn} holds.
For $n\in \N$ and $t\in [n, n+1)$, we set
\[
P(t)=T(t,n)P_n T(n,t).
\]
One can easily verify that $P(t)$ is a projection for each $t\ge 1$.
Take $t,s\geq 1$ and assume $t\geq s$. We choose $m, n\in \N$ such that $m\le t<m+1$ and $n\le s <n+1$. Clearly, $m\ge n$. Then,
\[
\begin{split}
T(t,s)P(s)&=T(t,s)T(s,n)P_nT(n, s)\\
&=T(t,n)P_n T(n,s) =T(t,m)\cA(m,n)P_n T(n,s)\\
&=T(t,m)P_m\cA(m,n) T(n,s)=T(t,m)P_m T(m,t)T(t,s)\\
&=P(t)T(t,s).
\end{split}
\]
In particular, \eqref{Pro} holds. 

Moreover, the previous calculations 
also show that for $x\in \R^d$,
\[
\begin{split}
\|T(t,s)P(s)x \|_t &=\| T(t,m)\cA(m,n)P_n T(n,s)x\|_t\\
&\leq K \left(\frac{t}{m}\right)^a K\left(\frac{m}{n}\right)^{-\lambda} K\left(\frac{s}{n}\right)^a\|x\|_s     \\
&\le K^3 4^a \bigg (\frac m n\bigg )^{-\lambda}\|x\|_s 
\le K^34^a2^\lambda \bigg (\frac t s \bigg )^{-\lambda}\|x\|_s
\\
\end{split}
\]
by~\eqref{pd1sn} and~\eqref{bnd-norm},
which yields~\eqref{d1}. Similarly, one can establish~\eqref{d2}. Noting that \eqref{bnd-norm} is satisfied by the hypothesis, it follows that \eqref{LDE} admits a strong polynomial dichotomy with respect to the family of norms $\mathcal{S}$ as claimed. This completes the proof.
\end{proof}

\subsection{Polynomial dichotomy spectrum} Given a family of norms $\mathcal{S}=\{\| \cdot \|_t; \ t\ge 1 \}$ on $\R^d$,
 let us consider the set $\Sigma_{PD, A(\cdot),\mathcal{S}}$ which consists of all $\tau\in \R$ such that the equation
\begin{equation}\label{LDE2}
x'=\bigg (A(t)-\frac \tau t \Id \bigg)x, \quad t\ge 1
\end{equation}
does not admit a strong polynomial dichotomy with respect to the family of norms $\mathcal{S}$. 
\begin{remark}
Observing that the evolution family of~\eqref{LDE2} is given by
\[
T_\tau(t,s)=\bigg (\frac t s\bigg )^{-\tau}T(t,s),
\]
we have that 
\[
T_\tau(n+1,n)=\bigg (\frac{n+1}{n}\bigg)^{-\tau}T(n+1,n), \quad n\in \N.
\]
 Using this together with Proposition  \ref{prop: equiv pol dich contXdisc},  we see that if $T(t,s)$ satisfies \eqref{bnd-norm} with respect to the family of norms $\mathcal{S}=\{\| \cdot \|_t; \ t\ge 1 \}$, then the following properties are equivalent:
\begin{itemize}
\item the sequence $((\frac{n+1}{n})^{-\tau}T(n+1,n))_{n\in \N}$ does not admit a strong polynomial dichotomy with respect to the sequence of norms $\{\| \cdot \|_n; \ n\in \N\}$;
\item \eqref{LDE2} does not admit a strong polynomial dichotomy with respect to the family of norms $\mathcal{S}$.
\end{itemize}
\end{remark}

The following result is a direct consequence of the above remark.

\begin{corollary}\label{cor: equal spectrum cont}
Let $\mathcal{S}=\{\| \cdot \|_t; \ t\ge 1 \}$ be a family of norms on $\R^d$
 and consider $\mathcal{S}'=\{\| \cdot \|_n; \ n\in \N\}$. Assume that there exist $K, a>0$ such that~\eqref{bnd-norm} holds. Then,
\[
\Sigma_{PD, A(\cdot), \mathcal{S}}=\Sigma_{PD, \mathbb A,\mathcal{S}'},
\]
where the sequence $\mathbb A=(A_n)_{n\in \N}$ is given by~\eqref{AN}. In particular, $\Sigma_{PD, A(\cdot),\mathcal{S}}$ is given by~\eqref{PDA} with $1\leq r\le d$ and $0<a_1\leq b_1<a_2\leq \ldots < a_r\leq b_r$.
\end{corollary}

\subsection{Smooth linearization} The following is a version of Theorem \ref{theo: main linearization discrete} in the continuous time setting.
\begin{theorem}\label{theo: linarization continuous time pol}
Let~\eqref{LDE} admit a nonuniform strong polynomial dichotomy and let $\mathcal{S}=\{\| \cdot \|_t; \ t\ge 1\}$ be a family of norms given by Proposition {\rm \ref{prop: equiv nonunifXnorms cont}}. Assume that $\Sigma_{PD, A(\cdot),\mathcal{S}}$ is given by~\eqref{PDA}, where $a_i, b_i$ satisfy~\eqref{sp1} and~\eqref{sp2}.
 Furthermore, suppose that $f\colon [1, \infty)\times \R^d \to \R^d$ is a $C^1$ map satisfying the following conditions:
\begin{itemize}
\item for $t\ge 1$,
\begin{equation}\label{1950}
f(t,0)=0 \quad \text{and} \quad D_x f(t,0)=0;
\end{equation}
\item there exists $\eta >0$ such that 
\begin{equation}\label{1951}
\| D_x f(t,x)\| \le \frac{\eta}{t^{1+4\varepsilon}}, \quad \text{for $x\in \R^d$;}
\end{equation}
\item there exists $L>0$ such that 
\begin{equation}\label{1952}
\|D_x f(t,x)-D_x f(t, y)\| \le \frac{L}{t^{1+5\varepsilon}} \|x-y\|, \quad \text{for $x, y\in \R^d$,}
\end{equation}
where $\varepsilon\geq 0$ is given in Definition {\rm \ref{def: nonunif pol dich continuous}}.
\end{itemize}
Then, provided that $\eta$ is sufficiently small,  there exist $C^1$ maps $H, G\colon [1, \infty) \times \R^d \to \R^d$ with the following properties:
\begin{itemize}
\item[i)] if $t\mapsto x(t)$ is a solution of~\eqref{NDE}, then $t\mapsto H(t, x(t))$ is a solution of~\eqref{LDE};
\item[ii)] if $t\mapsto x(t)$ is a solution of~\eqref{LDE}, then $t\mapsto G(t, x(t))$ is a solution of~\eqref{NDE};
\item[iii)] for $t\ge 1$ and $x\in \R^d$, 
\[
H(t, G(t,x))=x \quad \text{and} \quad G(t, H(t,x))=x;
\]
\item[iv)] there exist $R, \zeta>0$ such that 
\[
\|D_xH(t,x)\|\leq Rt^{4\varepsilon} \text{ and } \|D_xG(t,x)\| \leq Rt^{4\varepsilon}, 
\]
for every $t\geq 1$ and $x\in \R^d$ satisfying $\|x\|\leq \frac{\zeta}{t^{3\varepsilon}}$.
\end{itemize}
\end{theorem}

\begin{proof} The general idea of the proof consists of using the discretization of \eqref{LDE} given by \eqref{AN} in order to obtain Theorem \ref{theo: linarization continuous time pol} as a consequence of Theorem \ref{theo: main linearization discrete}.

Let $(A_n)_{n\in \N}$ be the sequence of operators on $\R^d$ given by~\eqref{AN}. Combining our assumptions with Proposition \ref{prop: equiv pol dich contXdisc} and Corollary \ref{cor: equal spectrum cont}, we understand that $(A_n)_{n\in \N}$ satisfies the hypothesis of Theorem \ref{theo: main linearization discrete}. Let $t\mapsto \varphi (t, t_0; x_0)$ denote the solution of~\eqref{NDE} satisfying $x(t_0)=x_0$. Hence, by the variation of constants formula, we have that 
\begin{align}\label{def-phi}
\phi(t, n; x)=T(t,n)x+\int_n^t T(t,r)f(r, \phi(r, n;x))\, dr.
\end{align}
For $n\in \N$ and $x\in \R^d$, set 
\begin{equation}\label{eq: def gn cont case}
g_n(x)=\phi(n+1, n;x)-A_n x=\int_n^{n+1}T(n+1, r)f(r, \phi(r,n;x))\, dr.
\end{equation}
We will now verify that this sequence of $C^1$ maps $g_n:\mathbb{R}^d\to \mathbb{R}^d$, $n\in \N$ satisfies all assumptions of Theorem~\ref{theo: main linearization discrete}.

It follows from~\eqref{1950} that  $\phi(t, t_0;0)=0$, and thus $g_n(0)=0$ for $n\in \N$. Moreover, 
\[
\begin{split}
Dg_n(0)
&=\int_n^{n+1}T(n+1, r)D_xf(r, 0)D_x \phi(r,n;0)\, dr =0,
\end{split}
\]
for $n\in \N$. We conclude that~\eqref{zero} holds.

 Observe that 
\begin{equation}\label{dphi}
D_x\phi(t, n; x)=T(t,n)+\int_n^t T(t,r)D_xf(r, \phi(r, n;x))D_x\phi(r, n;x)\, dr
\end{equation}
for $t\ge n$ and $x\in \R^d$.  In particular, from~\eqref{bnd} and~\eqref{1951} we obtain that for every $t\in [n, n+1]$ and $x\in \R^d$,
\[
\begin{split}
\lVert D_x\phi(t, n; x) \rVert&\le K\bigg (\frac tn \bigg )^a n^\varepsilon + \int_n^t K \bigg (\frac t r\bigg )^ar^\varepsilon \frac{\eta}{r^{1+4\varepsilon}} \lVert D_x\phi(r, n;x)\rVert\, dr
\\
&\le K2^an^\varepsilon +K\eta 2^a \int_n^t  \lVert D_x\phi(r, n;x)\rVert\, dr.
\end{split}
\]
 Hence, it follows from Gronwall's lemma that 
$
\lVert D_x\phi(t, n; x) \rVert \le K2^an^\varepsilon e^{\int_n^tK\eta 2^a \, dr},
$
and therefore 
\begin{equation}\label{910x}
\lVert D_x\phi(t, n; x) \rVert \le \hat M n^\varepsilon
\end{equation}
for every $t\in [n, n+1]$ and $x\in \R^d$, where
$
\hat M:=K2^ae^{K\eta 2^a}.
$
On the other hand, note that
\begin{equation}\label{Dfm}
Dg_n(x)=\int_n^{n+1}T(n+1, r)D_xf(r, \phi(r, n;x))D_x\phi(r, n;x)\, dr.
\end{equation}
Then, combining \eqref{bnd}, \eqref{1951}, \eqref{910x} with \eqref{Dfm} we get
\[
\lVert Dg_n(x) \rVert \le \int_n^{n+1}K \bigg (\frac{n+1}{r} \bigg )^a r^\varepsilon \frac{\eta}{r^{1+4\varepsilon}} \hat M n^\varepsilon \, dr \le \frac{K\hat M2^{a+1+2\varepsilon}\eta }{(n+1)^{1+2\varepsilon}},
\]
for each $n\in \N$ and $x\in \R^d$. We conclude that \eqref{non1} holds with $c:=K\hat M2^{a+1+2\varepsilon}\eta >0$, which can be sufficiently small by taking $\eta$ small enough.

Next, we observe that~\eqref{dphi} implies that 
\[
\begin{split}
&D_x\phi(t, n;x)-D_x\phi(t, n; y)
\\
&=\int_n^t T(t,r)D_xf(r, \phi(r, n;x))D_x\phi(r, n;x)\, dr \\
&\phantom{=}-\int_n^t T(t,r)D_xf(r, \phi(r, n;y))D_x\phi(r, n;y)\, dr \\
&=\int_n^t T(t,r)D_xf(r, \phi(r, n;x))(D_x\phi(r, n;x)-D_x\phi(r, n;y))\, dr \\
&\phantom{=}+\int_n^t T(t,r)(D_xf(r, \phi(r, n;x))-D_xf(r, \phi(r, n;y)))D_x\phi(r, n;y)\, dr.
\end{split}
\]
Hence, it follows from~\eqref{bnd}, \eqref{1951}, \eqref{1952} and~
\eqref{910x} (together with the
mean-value theorem) that
\[
\begin{split}
&\lVert D_x\phi(t, n;x)-D_x\phi(t, n; y)  \rVert
\\
&\le \int_n^t K \bigg (\frac  t r\bigg )^a r^\varepsilon\frac{L}{r^{1+5\varepsilon}} \lVert \phi(r, n;x)-\phi(r, n;y)\rVert \hat M n^\varepsilon \, dr\\
&\phantom{\le}+\int_n^t K \bigg (\frac  t r\bigg )^ar^\varepsilon \frac{\eta}{r^{1+4\varepsilon}} \lVert D_x\phi(r, n;x)-D_x\phi(r, n;y)\rVert \, dr \\
&\le KL2^a\hat M^2\|x-y\|+K2^a \eta \int_n^t\lVert D_x\phi(r, n;x)-D_x\phi(r, n;y)\rVert \, dr,
\end{split}
\]
for $t\in [n, n+1]$ and $x, y\in \R^d$.   By Gronwall's inequality again, one can conclude that there exists $\check M>0$ such that
\begin{equation}\label{Df-Df}
\lVert D_x\phi(t, n;x)-D_x\phi(t, n; y)  \rVert  \le  \check M\lVert x-y\rVert,
\end{equation}
for $n\in \N$, $t\in [n, n+1]$ and $x, y\in \R^d$.

Moreover, since
\[
\begin{split}
&Dg_{n}(x)-Dg_{n}(y)
\\
&=\int_n^{n+1}T(n+1, r)D_xf(r, \phi(r, n;x))D_x\phi(r, n;x)\, dr \\
&\phantom{=}-\int_n^{n+1}T(n+1, r)D_xf(r, \phi(r, n;y))D_x\phi(r, n;y)\, dr \\
&=\int_n^{n+1}T(n+1, r)D_xf(r, \phi(r, n;x))(D_x\phi(r, n;x)-D_x\phi(r, n;y))\, dr \\
&\phantom{=}+\int_n^{n\!+\!1}T(n+1, r)(D_xf(r, \phi(r, n;x))
\!-\!D_xf(r, \phi(r, n;y)))D_x\phi(r, n;y)\, dr, \\
\end{split}
\]
we obtain from \eqref{1951}, \eqref{1952},  \eqref{910x} and \eqref{Df-Df} (together with the mean-value
theorem) that
\[
\begin{split}
&\lVert Dg_{n}(x)-Dg_{n}(y)  \rVert
\\
&\le \int_n^{n+1}K \bigg (\frac{n+1}{r} \bigg )^ar^\varepsilon \frac{\eta}{r^{1+4\varepsilon}} \check M\lVert x-y\rVert\, dr \\
&\phantom{\le}+\int_n^{n+1}K \bigg (\frac{n+1}{r} \bigg )^a r^\varepsilon \frac{L}{r^{1+5\varepsilon}} \hat M n^\epsilon \lVert x-y\rVert \hat M n^\varepsilon \, dr \\
&\le \frac{K\check M 2^{a+1+2\varepsilon}\eta}{(n+1)^{1+2\varepsilon}} \|x-y\| +\frac{K L \hat M^2 2^{a+1+2\varepsilon} }{(n+1)^{1+2\varepsilon}} \|x-y\|,\\
\end{split}
\]
for $n\in \N$ and $x, y\in \R^d$. Then, we conclude that \eqref{non2} holds, and therefore all the hypothesis of Theorem \ref{theo: main linearization discrete} are satisfied for $\mathbb{A}=(A_n)_{n\in \N}$ and $(g_n)_{n\in \N}$ defined above. 

Thus, provided that $\eta$ is sufficiently small, by Theorem \ref{theo: main linearization discrete} there exists  a sequence $(\psi_n)_{n\in \N}$ of $C^1$-diffeomorphisms $\psi_n \colon \R^d \to \R^d$ satisfying~\eqref{E} and~\eqref{EE} (for some $\tilde{M}, \tilde{\rho} >0$).
We now set 
\begin{equation}\label{eq: def of H and G}
    H(t,x)=T(t,n) \psi_n (\phi(n, t;x)), \quad  G(t,x)=\phi(t, n; \psi_n^{-1}(T(n,t)x)),
\end{equation}
for $n\in \N$, $t\in [n, n+1)$ and $x\in \R^d$. In order to verify that $H(t,x)$ and $G(t,x)$ satisfy properties i) and ii) of our statement, we see from \eqref{eq: def gn cont case} that 
\[
\varphi(n+1,n;x)=A_nx+g_n(x)
\]
for any $n\in \N$ and $x\in \R^d$.
Combining this fact with \eqref{E} we get that
\[ 
\begin{split}
H(t,\varphi(t,1;x))&=T(t,n) \psi_n (\phi(n, t;\varphi(t,1;x)))\\
&=T(t,n) \psi_n (\phi(n, 1;x)))\\
&=T(t,n) \psi_n (\phi(n,n-1; \varphi(n-1,1;x)))\\
&=T(t,n) \psi_n ((A_{n-1}+g_{n-1})\circ(\varphi(n-1,1;x)))\\
&=T(t,n) A_{n-1} \circ \psi_{n-1} (\varphi(n-1,1;x)))\\
&=T(t,n-1) \psi_{n-1} (\varphi(n-1,1;x))),\\
\end{split}
\]
for every $t\in[n,n+1)$ and  $n>1$. In the above formula we note that 
$$
T(t,n) \psi_n (\phi(n, 1;x)))=
T(t,n-1) \psi_{n-1} (\varphi(n-1,1;x)).
$$
Then, proceeding recursively we can conclude that 
\[
H(t,\varphi(t,1;x))=T(t,1) \psi_{1} (x) 
\]
for $t\geq 1$. In particular, $H(t,\varphi(t,1;x))$ is a solution of \eqref{LDE} proving i). Similarly, we can prove that $G$ satisfies property ii) of our statement. 

Now, observing that
\[
\begin{split}
H(t, G(t, x)) &=T(t, n)\psi_n(\phi(n, t; G(t,x))) \\
&=T(t,n)\psi_n (\phi(n, t; \phi(t, n;\psi_n^{-1}(T(n, t)x)) ) \\
&=T(t,n)\psi_n (\psi_n^{-1}(T(n,t)x)) \\
&=T(t,n)T(n, t)x =x,
\end{split}
\]
for each $x\in \R^d$, $t\in [n, n+1)$ and $n\in \N$, we get that $H(t, G(t, x))=x$ for every $t\geq 1$ and $x\in \R^d$. Similarly, we can show that $G(t, H(t, x))=x$ for every $t\geq 1$ and $x\in \R^d$. Consequently, conclusion iii) of our statement holds true. It remains to show that iv) is satisfied. 

Proceeding as we did to obtain~\eqref{910x}, we find that there exists $\hat M>0$ such that
\begin{equation}\label{Dx-Inv}
\|D_x\phi(n,t;x) \| \le  \hat M n^\varepsilon 
\end{equation}
for $n\in \N$, $t\in [n, n+1)$ and $x\in \R^d$. In particular, \eqref{Dx-Inv} implies that
\begin{equation}\label{bx1}
\|\phi(n,t;x) \| \le \hat M n^\epsilon \|x\|.
\end{equation}
Therefore, if $\|x\|\leq \frac{\tilde \rho}{\hat M t^{3\varepsilon}}$ we get that $\|\phi(n,t;x) \| \leq \frac{\tilde{\rho}}{n^{2\varepsilon}}$ for every $t\in[n,n+1)$ and $n\in \N$. Since for every
$t\ge 1$ there is an $n\in\mathbb{N}$ such that $t\in [n,n+1)$, combining these observations with \eqref{EE}, \eqref{bnd} and the definition of $H$, we get that 
\[
\begin{split}
\|D_xH(t,x)\|&\leq\|T(t,n)\| \|D_x\psi_n (\phi(n, t;x))\|\|D_x\phi(n, t;x)\|\\
&\leq K2^an^\varepsilon \tilde{M}n^{2\varepsilon} \hat Mn^\varepsilon
\leq 2^aK\tilde{M}\hat M t^{4\varepsilon},
\end{split}
\]
whenever $\|x\|\leq \frac{\tilde \rho}{\hat M t^{3\varepsilon}}$. Similarly,
\[
\begin{split}
\|D_xG(t,x)\|&\leq 2^aK\tilde{M}\hat M t^{4\varepsilon},\\
\end{split}
\]
for every $t\geq 1$ whenever $\|x\|\leq \frac{\tilde \rho}{\hat M t^{3\varepsilon}}$. Therefore, taking $R=2^aK \tilde{M}\hat M$ and $\zeta =\frac{\tilde{\rho}}{\hat M}$, we obtain iv). This concludes the proof of Theorem \ref{theo: linarization continuous time pol}.
\end{proof}

\begin{example} The following is a continuous time version of the Example \ref{example: main theo discrete}. Let $A,P:[1,+\infty)\to \R^{2\times 2}$ be given by
\[
A(t)=\begin{pmatrix}
-\frac{1}{t} & 0\\
0 & \frac{1}{t}
\end{pmatrix} \text{ and } P(t)=\begin{pmatrix}
1 & 0\\
0 & 0
\end{pmatrix}
\]
and consider the associated dynamical system given by \eqref{LDE}. Then, the evolution family associated to this system is
\[
T(t,s)=\begin{pmatrix}
\frac{s}{t} & 0\\
0 & \frac{t}{s}
\end{pmatrix} \text{ for every } t,s\geq 1.
\]
We can easily see that \eqref{LDE} admits a nonuniform strong polynomial dichotomy with constants $K=a=\lambda=1$, $\varepsilon=0$ and projections $P(t)$, $t\geq 1$. In fact, it admits a strong polynomial dichotomy with respect to the (constant) family of norms $\mathcal{S}=\{\|\cdot\|_t; \ t\ge 1\}$, where $\| \cdot \|_t=\| \cdot \|$ is the Euclidean norm. Moreover, the discretization of this system defined by \eqref{AN} is given precisely by $\mathbb{A}=(A_n)_{n\in \N}$ and $(P_n)_{n\in \N}$ from the Example \ref{example: main theo discrete}. Therefore, by Corollary \ref{cor: equal spectrum cont}, 
$$
\Sigma_{PD,A(\cdot),\mathcal{S}}=\Sigma_{PD,\mathbb{A},\mathcal{S}'}=\{-1,1\}
$$ 
and $r=2$, $a_1=b_1=\frac{1}{2}$ and $a_2=b_2=2$. Consequently, conditions \eqref{sp1} and \eqref{sp2} are satisfied. 

Considering $f:[1,+\infty)\times \R^2\to \R^2$ given by
\[f(t,(x_1,x_2))=\frac{\eta}{t+1}(\xi(x_1),\xi(x_2)),\]
where $\eta>0$ is a constant and $\xi$ is as in Example \ref{example: main theo discrete}, it follows that conditions \eqref{1950}, \eqref{1951} and \eqref{1952} are satisfied. In particular, Theorem \ref{theo: linarization continuous time pol} may be applied whenever $\eta>0$ is small enough. Moreover, as in the discrete time case, we can easily verify that \eqref{LDE} does not admit nonuniform strong exponential dichotomy.
\end{example}

\subsection{Differentiable and H\"older linearization} The following is a version of Theorem \ref{theo: diff+holder linearization} in the continuous time case. 

\begin{theorem}\label{theo: diff+holder linearization continuous}
Let \eqref{LDE}, $\mathcal{S}=\{\| \cdot \|_t; \ t\ge 1\}$ and $\Sigma_{PD, A(\cdot),\mathcal{S}}$ be as in Theorem {\rm\ref{theo: linarization continuous time pol}} such that \eqref{sp1} and~\eqref{sp3} are satisfied. Assume that $f\colon [1, \infty)\times \R^d \to \R^d$ is a $C^1$ map satisfying conditions \eqref{1950}-\eqref{1952}. Then, provided that $\eta$ is sufficiently small, there exist continuous maps $H, G\colon [1, \infty) \times \R^d \to \R^d$ satisfying properties i), ii) and iii) from Theorem {\rm \ref{theo: linarization continuous time pol}}.
Moreover, there exist constants $\tilde R,\tilde \zeta>0$ such that 
\begin{equation}\label{eq: H is diff 0 cont}
H(t,x)=x+t^{\varepsilon (4+3\varrho)}o(\|x\|^{1+\varrho}),\quad
G(t,x)=x+ t^{\varepsilon (4+3\varrho)}o(\|x\|^{1+ \varrho})
\end{equation}
as $\|x\|\to 0$ and
\begin{align}\label{eq: H and G are holder}
\begin{split}
&\|H(t,x)-H(t,y)\| \leq \tilde Rt^{\varepsilon(1+3\alpha_1)}\|x-y\|^{\alpha_1},
\\
&\|G(t,x)-G(t,y)\| \leq  \tilde R t^{\varepsilon(1+3\alpha_1)} \|x-y\|^{\alpha_1}
\end{split}
\end{align}
for $x,y\in\mathbb{R}^d$ satisfying $\|x\|,\|y\|\le \frac{\tilde \zeta}{t^{3\varepsilon}}$, where $\varrho$ and $\alpha_1$ are given in Theorem {\rm \ref{theo: main linearization discrete}}.
\end{theorem}
\begin{proof}
We may proceed exactly as we did in the proof of Theorem \ref{theo: linarization continuous time pol}, but apply Theorem \ref{theo: diff+holder linearization} instead of apply Theorem \ref{theo: main linearization discrete} in this proof. 

Using the expressions for $H$ and $G$ given in \eqref{eq: def of H and G} and combining \eqref{001}, \eqref{bnd}, \eqref{910x} with \eqref{Dx-Inv}, we obtain \eqref{eq: H and G are holder}. We now show how to establish the first estimate in \eqref{eq: H is diff 0 cont}. The second one is similar. Using \eqref{000}, \eqref{bnd} and \eqref{bx1} we get that for every $n\in \N$ and $t \in [n,n+1)$,
\begin{align}
\lVert H(t,x)-x\rVert
&=\lVert T(t,n)\psi_n(\varphi(n,t;x))-x\rVert
\nonumber\\
&\le
\lVert T(t,n)\psi_n(\varphi(n,t;x)) -T(t,n)\varphi(n,t;x)\rVert
\nonumber\\
& 
\phantom{\le}+\lVert T(t,n)\varphi(n,t;x)-T(t,n)T(n,t)x\rVert
\nonumber\\
&\le
K2^at^{\varepsilon} t^{2\varepsilon (1+\varrho)} o(\lVert \varphi(n, t;x)\rVert^{1+ \varrho})
+K2^at^{\varepsilon} \lVert \phi(n,t;x)-T(n,t)x\rVert
\nonumber\\
&\le
t^{\varepsilon +3\varepsilon (1+\varrho)}o(\lVert x\rVert^{1+ \varrho})
+K2^at^{\varepsilon}\lVert \phi(n,t;x)-T(n,t)x\rVert
\nonumber
\end{align}
as $\|x\|\to 0$. Moreover, it follows by \eqref{1950}, \eqref{1952}, \eqref{def-phi} and \eqref{bx1} that for every $n\in \N$ and $t \in [n,n+1)$,
\[
\begin{split}
\lVert \varphi(n,t;x)-T(n,t)x\rVert
&\le \int_n^{n+1}\lVert T(n,s)f(s, \varphi(n,s;x))\rVert\, ds
\\
&\le \int_n^{n+1} K2^a s^{\varepsilon}  \sup_{\theta\in (0,1)}\|D_xf(s, \theta \varphi(n,s;x))\|\|\varphi(n,s;x)\|\,ds
\\
&\le \int_n^{n+1} K2^a s^{\varepsilon} \frac{L}{s^{1+5\varepsilon}}
\|\varphi(n,s; x)\|^2 ds
\\
&\le \int_n^{n+1}K2^a \frac{L}{s^{1+4\varepsilon}} \hat M ^2s^{2\varepsilon} \|x\|^2 ds
\le K2^aL \hat M ^2 \lVert x\rVert^2.
\end{split}
\]
Combining these observations we get that
\[
\begin{split}
H(t,x)&=x+t^{\varepsilon +3\varepsilon (1+\varrho)} o(\lVert x\rVert^{1+ \varrho})+t^\varepsilon O(\|x\|^2)=x+ t^{\varepsilon (4+3\varrho)}o(\lVert x\rVert^{1+ \varrho})
\end{split}
\]
as $\|x\|\to 0$. This concludes the proof of the theorem.
\end{proof}

\section{Appendix A}\label{sec: appendix a}

Let $\mathbb B=(B_n)_{n\in \Z}$ be a sequence of invertible operators on $\R^d$ that  admits a strong exponential dichotomy with respect to a sequence of norms $\mathcal S':=\{ \|\cdot \|_n; \ n\in \Z\}$. 
Set
\[
Y_\infty:=\bigg \{ \mathbf x=(x_n)_{n\in \Z} \subset \R^d: \|\mathbf x\|_\infty:=\sup_{n\in \Z} \|x_n\|_n <+\infty \bigg \}.
\]
Then, $(Y_\infty, \| \cdot \|_\infty)$ is a Banach space. We further define a linear operator ${B}^*\colon Y_\infty \to Y_\infty$  by
\[
({B}^* \mathbf x)_n:=B_{n-1}x_{n-1}, \quad \forall n\in \Z, \quad \forall \mathbf x=(x_n)_{n\in \Z}\in Y_\infty.
\]
Define the spectrum
$$
\sigma({B}^*):=\{\varrho\in \mathbb C: \varrho\Id-{B}^* ~\mbox{is not invertible on $Y_\infty$}\}.
$$
By \cite[Lemma 2]{DZZ}, we see that $\varrho\in \sigma({B}^*)$ (actually considering the complexification of ${B}^*$ and $Y_\infty$)
if and only if $\left(\frac{1}{|\varrho|}B_n\right)_{n\in \Z}$ does not admit a strong exponential dichotomy with respect to the sequence of norms $\mathcal S'$. This fact shows (see Definition~\ref{def-SED} and Remark~\ref{aaa}) that
$$
|\sigma({B}^*)|=\Sigma_{ED, \mathbb B, \mathcal S'},
$$
where 
\[
|\sigma({B}^*)|=\{ |\lambda|; \ \lambda \in \sigma({B}^*)\}.
\]
Then, we have the following theorem.
\begin{theorem}\label{Ap}
Suppose that 
\begin{align*}
\Sigma_{ED, \mathbb B,  {\mathcal S}'}=\bigcup_{i=1}^r [a_i, b_i],
\end{align*}
where $a_i$ and $b_i$
satisfy \eqref{sp1}-\eqref{sp2}. Moreover, let $f_n\colon \R^d \to \R^d$ be a sequence of $C^1$ maps such that \eqref{fzero}-\eqref{fm} hold for $n\in\mathbb{Z}$.
Then, provided that $\eta$ is sufficiently small,  there exists a sequence $h_n\colon \R^d\to \R^d$ of $C^1$-diffeomorphisms such that
\begin{itemize}
\item for $n\in \Z$,
\begin{align}\label{con-hn}
h_{n+1} \circ (B_n+f_n)=B_n\circ h_n;
\end{align}
\item there exist $M, \rho>0$ such that
\begin{equation}\label{conj}
\|D h_n(v)z\|_n \le M\|z\|_n \quad \text{and} \quad \|Dh_n^{-1}(v)z\|_n \le M\|z\|_n, \quad \forall n\in \Z,
\end{equation}
for all  $z\in \R^d$ and $v\in\mathbb{R}^d$ such that $\|v\|_n\le \rho$.
\end{itemize}
\end{theorem}

\begin{remark}
This result is essentially established in~\cite[Theorem 2]{DZZ} except that the conclusion~\eqref{conj} was not explicitly written.  Notice that such $C^1$ conjugacy $h_n$ without satisfying \eqref{conj} always exists by a recursive construction.
\end{remark}

{\it Proof of Theorem} \ref{Ap}.
First of all, by the above discussion, we understand that
$$
|\sigma({B}^*)|=\Sigma_{ED, \mathbb B, \mathcal S'}
=\bigcup_{i=1}^r [a_i, b_i]
$$
Next, define $F\colon Y_\infty \to Y_\infty$ by
\[
(F(\mathbf x))_n:=B_{n-1}x_{n-1}+f_{n-1}(x_{n-1}), \quad \forall n\in \Z, \ \forall \mathbf x=(x_n)_{n\in \Z}\in Y_\infty.
\]
According to the proof of \cite[Theorem 2]{DZZ}, we get the following conclusions:
\begin{itemize}
\item $F\in C^{1,1}$ and $DF(\mathbf 0)=B^*$;
\item there exists $C>0$ such that $\| DF(\mathbf x)-B^*\| \le C\eta$ for $\mathbf x\in Y_\infty$.
\end{itemize}
Hence, provided that $\eta$ is sufficiently small, it follows from the $C^1$ linearization result given in~\cite[Appendix]{DZZ} that there exists a $C^1$-diffeomorphism  $\Phi \colon Y_\infty \to Y_\infty$ such that
\begin{equation}\label{P}
\Phi \circ F={B}^*\circ \Phi.
\end{equation}
For $v\in \R^d$ and $m\in \Z$, we set
\[
h_m(v):=(\Phi(\mathbf v^m))_m,\qquad \forall v\in \mathbb{R}^d,~\forall m\in\mathbb{Z},
\]
where $\mathbf v^m:=(v_n^m)_{n\in \Z}$ is given by $v_m^m=v$ and $v_n^m=0$ for $n\neq m$. Then, it is verified in the proof of~\cite[Theorem 2]{DZZ} that $h_m\colon \R^d \to \R^d$ are $C^1$-diffeomorphisms satisfying (\ref{con-hn}).
Moreover, we know that
\begin{equation}\label{DhDPhi}
  D h_m(v) z=(D\Phi(\mathbf v^m)\mathbf z^m)_m \quad \text{and} \quad Dh_m^{-1}(v)z=(D\Phi^{-1}(\mathbf v^m)\mathbf z^m)_m,
\end{equation}
for every $v, z\in \R^d$ and $m\in \Z$, 
where $\mathbf z^m$ is defined as $\mathbf v^m$, replacing $v$ by $z$.

In what follows, we prove \eqref{conj}.
It follows easily from~\eqref{P} that $\Phi(\mathbf 0)=\mathbf 0$.
Since $\Phi$ is $C^1$, there exists constants $M, \rho>0$ such that
\begin{equation}\label{P1}
\|D\Phi(\mathbf x)\| \le M \quad  \text{and} \quad  \|D\Phi^{-1}(\mathbf x)\|\le M, 
\end{equation}
for $\mathbf x\in Y_\infty$ satisfying $\|\mathbf x\|_\infty \le \rho$.
Then, we obtain from (\ref{DhDPhi}) and~\eqref{P1} that
$$
\|D h_m(v)z\|_m\le \|D\Phi(\mathbf v^m)\mathbf z^m\|_\infty
\le \|D\Phi(\mathbf v^m)\|\,\|\mathbf z^m\|_\infty
=M\|z\|_m,
$$
for $\|v\|_m \le \rho$ (which implies that $\|\mathbf v^m\|_\infty\le \rho$). Similarly,
\[
\|D h_m^{-1}(v)z\|_m\le M\|z\|_m,  \quad \text{for $\|v\|_m\le \rho$.}
\]
This proves (\ref{conj}) and the proof of the theorem is completed. $\hfill\Box$

\section{Appendix B}\label{A-B}
We now provide the proof of Proposition~\ref{22}. We stress that the proof is rather standard (see for example the proofs  of~\cite[Theorem 3.4]{AS} and~\cite[Theorem 5]{BDV}) and that it is based on the ideas developed in~\cite{SS-JDE78}. However, we provide complete details for the sake of completeness.

We will begin by establishing several auxiliary result.
\begin{lemma} \label{l1}
Let $(A_n)_{n\in \Z^+}$ be a sequence of invertible operators on $\R^d$ that admits a strong exponential dichotomy with respect to a sequence of norms $\| \cdot \|_n$, $n\in \Z^+$ and projections $P_n$, $n\in \Z^+$. Then,
\[
\Ima P_n=\bigg \{v\in \R^d: \sup_{m\ge n} \|\cA(m,n)v\|_m <+\infty \bigg \},
\]
for each $n\in  \Z^+$.
\end{lemma}

\begin{proof}
It follows from the first inequality in~\eqref{ed1sn} that \begin{equation}\label{sup}\sup_{m\ge n} \|\cA(m,n)v\|_m <+\infty, \end{equation} for each $v\in \Ima P_n$.
Take now $v\in \R^d$ satisfying~\eqref{sup}. Since $v=P_nv+Q_nv$, it follows from~\eqref{ed1sn} that
\begin{equation}\label{sup2}
\sup_{m\ge n} \|\cA(m,n)Q_n v\|_m <+\infty.
\end{equation}
On the other hand, the second estimate in~\eqref{ed1sn} implies that for $m\ge n$,
\[
\|Q_nv\|_n =\|\cA(n,m)\cA(m,n)Q_nv\|_n \le {K}e^{-\lambda (m-n)}\|\cA(m, n)Q_n v\|_m,
\]
and thus
\[
\frac 1 K e^{\lambda (m-n)}\|Q_n v\|_n \le \|\cA(m, n)Q_n v\|_m.
\]
From~\eqref{sup2} we conclude that $Q_n v=0$, and therefore $v\in \Ima P_n$.
\end{proof}

\begin{lemma}\label{L}
Let $(A_n)_{n\in \Z^+}$ be a sequence of invertible operators on $\R^d$ that admits a strong exponential dichotomy with respect to a sequence of norms $\| \cdot \|_n$, $n\in \Z^+$ and projections $P_n$, $n\in \Z^+$. Futhermore, let $(P_n')_{n\in \Z^+}$ be a sequence of projections on $\R^d$ such that
\begin{equation}\label{prime}
P_{n+1}'A_n=A_nP_n', \quad n\in \Z^+.
\end{equation}
Then, $(A_n)_{n\in \Z^+}$ admits a strong exponential dichotomy with respect to the sequence of norms $\| \cdot \|_n$, $n\in \Z^+$ and projections $P_n'$, $n\in \Z^+$ if and only if
$\Ima P_0=\Ima P_0'$.

\end{lemma}

\begin{proof}
If $(A_n)_{n\in \Z^+}$ admits a strong exponential dichotomy with respect to the sequence of norms $\| \cdot \|_n$, $n\in \Z^+$ and projections $P_n'$, $n\in \Z^+$, it follows from Lemma~\ref{l1} that
\[
\Ima P_0=\Ima P_0'=\bigg \{v\in \R^d: \sup_{m\ge 0} \|\cA(m,0)v\|_m <+\infty \bigg \}.
\]

Suppose now that $\Ima P_0=\Ima P_0'$. Then,
\[
P_0P_0'=P_0' \quad \text{and} \quad P_0'P_0=P_0,
\]
and consequently
\[
P_0-P_0'=P_0(P_0-P_0')=(P_0-P_0')Q_0,
\]
where $Q_0=\Id-P_0$.
By~\eqref{ed1sn}, we have that
\[
\begin{split}
\|\cA(n,0)(P_0-P_0')v\|_n &=\|\cA(n,0)P_0(P_0-P_0')v\|_n  \\
&\le {K}e^{-\lambda n}\|(P_0-P_0')Q_0v\|_0 \\
&\le KLe^{-\lambda n}\|Q_0v\|_0\\
&= KLe^{-\lambda n}\|\cA(0,m)\cA(m,0)Q_0v\|_0\\
&= KLe^{-\lambda n}\|\cA(0,m)Q_m\cA(m,0)v\|_0\\
&\le K^2Le^{-\lambda n-\lambda m}\|\cA(m,0)v\|_m,
\end{split}
\]
for $m, n\in \Z^+$ and $v\in \R^d$, where
\[
L:=\max_{\|v\|_0 \le 1} \|(P_0-P_0')v\|_0.
\]
Therefore,
\[
\begin{split}
\|\cA(n,m)P_m' v\|_n &\le \|\cA(n,m)P_mv\|_n+\|\cA(n,m)(P_m-P_m')v\|_n \\
&=\|\cA(n,m)P_m v\|_n+\|\cA(n,0)(P_0-P_0')\cA(0,m)v\|_n\\
&\le {K}e^{-\lambda (n-m)}\|v\|_m+K^2Le^{-\lambda n-\lambda m}\|v\|_m\\
&\le K'e^{-\lambda (n-m)}\|v\|_m,
\end{split}
\]
for $n\ge m$ and $v\in \R^d$, where $K'=K+K^2L>0$. Similarly, setting $Q_m'=\Id-P_m'$ we obtain that
\[
\begin{split}
\| \cA(n,m)Q_m'v\|_n &\le \|\cA(n,m)Q_mv\|_n+\|\cA(n,m)(P_m-P_m')v\|_n \\
&\le {K}e^{-\lambda (m-n)}\|v\|_m+\|\cA(n,0)(P_0-P_0')\cA(0,m)v\|_n \\
&\le {K}e^{-\lambda (m-n)}\|v\|_m+K^2Le^{-\lambda n-\lambda m}\|v\|_m \\
&\le K'e^{-\lambda (m-n)}\|v\|_m,
\end{split}
\]
for $m\ge n$ and $v\in \R^d$. We conclude that $(A_n)_{n\in \Z^+}$ admits a strong exponential dichotomy with respect to the sequence of norms $\| \cdot \|_n$, $n\in \Z^+$ and projections $P_n'$, $n\in \Z^+$.
\end{proof}

\begin{corollary}
Suppose that $(A_n)_{n\in \Z^+}$ is a sequence of invertible operators on $\R^d$ that admits a strong exponential dichotomy with respect to a sequence of norms $\|\cdot \|_n$, $n\in \Z^+$ and projections $P_n$, $n\in \Z^+$. Furthemore, let $Y\subset \R^d$ be an arbitrary  subspace such that
\begin{equation}\label{dec}
\R^d=\Ima P_0 \oplus Y.
\end{equation}
Let $P_0'\colon \R^d \to \Ima P_0$ be the projection associated to~\eqref{dec}. Furthermore, for $n\in \Z^+$ set
\[
P_n'=\cA(n,0)P_0'\cA(0,n).
\]
Then, $(A_n)_{n\in \Z^+}$ admits a strong exponential dichotomy with respect to the sequence of norms $\|\cdot \|_n$, $n\in \Z^+$ and projections $P_n'$, $n\in \Z^+$.
\end{corollary}

\begin{proof}
Clearly, the projections $P_n'$  satisfy~\eqref{prime}. Moreover, $\Ima P_0'=\Ima P_0$. The conclusion now follows directly from Lemma~\ref{L}.
\end{proof}

Let us now fix an arbitrary sequence $\mathcal S=\{\| \cdot \|_n; \ n\in \Z^+\}$ of norms on $\R^d$ and a sequence $\mathbb  A=(A_n)_{n\in \Z^+}$ of invertible operators on $\R^d$. Suppose that~\eqref{bg1sne} holds with some $K, a>0$.
For $r>0$ and $n\in \Z^+$, let
\[
S_r(n):=\bigg \{v\in \R^d: \sup_{m\ge n}(r^{-(m-n)}\|\cA(m,n)v\|_m)<+\infty \bigg \}.
\]
Clearly,
\[
A_n S_r(n)=S_r(n+1) \quad \text{for $n\in \Z^+$ and $r\in (0, \infty)$,}
\]
which implies that $\dim S_r(n)$ does not depend on $n$, and thus we  denote it by $\dim S_r$.  Moreover, observe that for $r_1<r_2$ we have that $S_{r_1}(n)\subset S_{r_2}(n)$ for $n\in \Z^+$.
\begin{lemma}
The set $\Sigma:=\Sigma_{ED, \mathbb A, \mathcal S}$ is closed. Moreover, for $r\in (0, \infty)\setminus \Sigma$ we have
\begin{equation}\label{snr}
S_r(n)=S_{r'}(n),
\end{equation}
for all $n\in \Z^+$ and $r'$ in some open interval around $r$.
\end{lemma}

\begin{proof}
Take an arbitrary $r\in (0, \infty)\setminus \Sigma$. Let $\mu\in \R$ be such that $r=e^\mu$. Hence, there exist $K', \lambda >0$ and a sequence of projections $P_n$, $n\in \Z^+$ satisfying~\eqref{pro} and such that
\[
e^{-\mu(m-n)}\| \cA(m,n)P_nv\|_m \le K'e^{-\lambda (m-n)}\|v\|_n,
\]
and
\[
e^{\mu(m-n)}\| \cA(n,m)Q_n v\|_n \le K'e^{-\lambda (m-n)} \|v\|_m,
\]
for $m\ge n$ and $v\in \R^d$. Hence, for $\mu'\in \R$ we have that
\[
e^{-\mu'(m-n)}\| \cA(m,n)P_nv\|_m \le {K'}e^{-(\lambda-\mu+\mu') (m-n)}\|v\|_n,
\]
and
\[
e^{\mu'(m-n)}\| \cA(n,m)Q_n v\|_n \le {K'}e^{-(\lambda+\mu-\mu') (m-n)} \|v\|_m,
\]
for $m\ge n$ and $v\in \R^d$. We conclude that $(\frac{1}{r'}A_n)_{n\in \Z^+}$ admits an exponential dichotomy with respect to the sequence of norms $\|\cdot \|_n$, $n\in \Z^+$ and projections $P_n$, $n\in \Z^+$ whenever $r'=e^{\mu'}$ and $|\mu-\mu'|<\lambda$. Moreover, Lemma~\ref{l1} implies that~\eqref{snr} holds.
\end{proof}

\begin{lemma}\label{259}
The following properties hold:
\begin{itemize}
\item for $r>0$ sufficiently large we have that $r\notin \Sigma$ and $S_r(n)=\R^d$ for $n\in \Z^+$;
\item for $r>0$ sufficiently small we have that $r\notin \Sigma$ and $S_r(n)=\{0\}$ for $n\in \Z^+$.
\end{itemize}
\end{lemma}

\begin{proof}
For $r>e^a$, it follows from~\eqref{bg1sne} that
\[
\|r^{-(m-n)}\cA(m,n)x\|_m \le {K}e^{-(\log r-a)(m-n)}\|x\|_n,
\]
for $m\ge n$ and $x\in \R^d$. We conclude that the sequence $(\frac 1 r A_n)_{n\in \Z^+}$ admits a strong exponential dichotomy with respect to the sequence of norms $\| \cdot \|_n$, $n\in \Z^+$ and projections $P_n=\Id$, $n\in \Z^+$. This together with Lemma~\ref{l1} implies the first conclusion. Similarly, one can establish the second conclusion.
\end{proof}

\begin{lemma}\label{258}
For $r_1, r_2\in (0, \infty)\setminus \Sigma$, $r_1<r_2$ the following properties are equivalent:
\begin{itemize}
\item $S_{r_1}(n)=S_{r_2}(n)$ for some $n\in \Z^+$ (and so for all $n\in \Z^+$);
\item $[r_1, r_2]\cap \Sigma=\emptyset$.
\end{itemize}
\end{lemma}

\begin{proof}
Assume first that $S_{r_1}(n)=S_{r_2}(n)$ for $n\in \Z^+$. By Lemmas~\ref{l1} and~\ref{L}, we have that sequences $(\frac{1}{r_1}A_n)_{n\in \Z^+}$ and $(\frac{1}{r_2}A_n)_{n\in \Z^+}$ admit a strong exponential dichotomy with respect to the sequence of norms $\|\cdot \|_n$, $n\in \Z^+$ and the same sequence of projections $P_n$, $n\in \Z^+$.
Therefore, there exist $K', \lambda >0$ such that
\begin{equation}\label{tt1}
\|r_i^{-(m-n)}\cA(m,n)P_n x\|_m \le K'e^{-\lambda (m-n)}\|x\|_n
\end{equation}
and
\begin{equation}\label{tt2}
\|r_i^{m-n}\cA(n,m)Q_mx\|_n \le  K'e^{-\lambda (m-n)}\|x\|_m,
\end{equation}
for $i\in \{1,2\}$, $m\ge n$ and $x\in \R^d$. For $r\in [r_1, r_2]$, it follows from~\eqref{tt1} for $i=1$ that
\[
\|r^{-(m-n)}\cA(m,n)P_n x\|_m \le K'e^{-\lambda (m-n)}\|x\|_n,
\]
for $m\ge n$ and $x\in \R^d$. Similarly, by~\eqref{tt2} for $i=2$ we have that
\[
\|r^{m-n}\cA(n,m)Q_mx\|_n \le  K'e^{-\lambda (m-n)}\|x\|_m,
\]
for $m\ge n$ and $x\in \R^d$. We conclude that $(\frac 1 rA_n)_{n\in \Z^+}$ admits a strong exponential dichotomy with respect to the sequence of norms $\|\cdot \|_n$, $n\in \Z^+$. Hence, $r\notin \Sigma$.

The converse implication can be established by arguing exactly as in the proof of~\cite[Lemma 1]{BDV}.
\end{proof}

We are now in a position to prove Proposition~\ref{22}.
\begin{proof}[Proof of Proposition~\ref{22}]
Since $\Sigma$ is closed, it a disjoint union of closed intervals. Assume that it contains $d+1$ closed disjoint intervals. Then, there exist $c_1, \ldots, c_d\in (0, \infty)\setminus \Sigma$ such that intervals
\[
(0, c_1), \ (c_1, c_2), \ldots, (c_{d-1}, c_d), \ (c_d, \infty)
\]
intersect $\Sigma$. By Lemma~\ref{258}, we have that
\[
\dim S_{c_1}<\dim S_{c_2} <\ldots <\dim S_{c_d}.
\]
Moreover, Lemmas~\ref{259} and~\ref{258} imply that $\dim S_{c_1}>0$ and $\dim S_{c_d}<d$. This yields a contradiction.
\end{proof}

\section{Acknowledgments}
The authors are ranked in alphabetic order and their contributions should be treated equally. L. Backes was partially supported by a CNPq-Brazil PQ fellowship under Grant No. 307633/2021-7; D. Dragi\v cevi\'c was partially supported by Croatian Science Foundation under the project IP-2019-04-1239 and by the University of Rijeka under the projects uniri-prirod-18-9 and uniri-pr-prirod-19-16; W. Zhang was partially supported by NSFC grants \#11922105 and \#11831012.


\begin{thebibliography}{99}

\bibitem{AS} B. Aulbach and S. Siegmund, \emph{The dichotomy spectrum for noninvertible systems of linear
difference equations}, J. Differ. Equ. Appl. \textbf{7} (2001), 895--913.

\bibitem{AW} B. Aulbach and T. Wanner, \emph{Topological simplification of nonautonomous difference equations}, J. Differ. Equ. Appl. \textbf{12} (2006), 283–296.

\bibitem{BD} L. Backes and D. Dragi\v cevi\' c, \emph{Smooth linearization of nonautonomous coupled systems}, preprint, https://arxiv.org/abs/2202.12367.

\bibitem{BDV0}L. Barreira, D. Dragi\v cevi\' c and C. Valls, \emph{From one-sided dichotomies to two-sided dichotomies}, Discrete Contin. Dyn. Syst. \textbf{35} (2015), 2817--2844.

\bibitem{BDV1}L. Barreira, D. Dragi\v cevi\' c and C. Valls, \emph{Characterization of strong exponential dichotomies}, Bull. Braz. Math. Soc. \textbf{46} (2015), 81--103.

\bibitem{BDV} L. Barreira, D. Dragi\v cevi\' c and C. Valls, \emph{Nonuniform Spectrum on the Half Line and Perturbations}, Results. Math. \textbf{72} (2017), 125--143.

\bibitem{BFVZ} L. Barreira, M. Fan, C. Valls and J. Zhang, \emph{Robustness of nonuniform polynomial dichotomies for difference equations}. Topological Methods in Nonlinear Analysis, \textbf{37} (2011), 357--376.


\bibitem{BV-book} L. Barreira and C. Valls, Stability of nonautonomous differential equations \textbf{1926}, Berlin: Springer, 2008.



\bibitem{BV-NA09} L. Barreira and C. Valls, \emph{Polynomial growth rates}, Nonlinear Anal. \textbf{71} (2009), 5208--5219.


\bibitem{Bel73}
G. R. Belitskii, \emph{Functional equations and the conjugacy of diffeomorphism of finite smoothness class},
Funct. Anal. Appl. {\bf 7} (1973), 268--277.



\bibitem{BS1} A. J.  Bento and C. Silva, \emph{Stable manifolds for nonuniform polynomial dichotomies}, J. Funct. Anal. \textbf{257} (2009), 122--148.
\bibitem{BS2} A. J. Bento and C. Silva, \emph{table manifolds for nonautonomous equations with nonuniform polynomial dichotomies}, Q. J. Math. \textbf{63} (2012), 275--308.


\bibitem{CMR}  A. Casta\~{n}eda, P. Monzon and G. Robledo,
\emph{Nonuniform contractions and density stability results via a smooth topological equivalence},
preprint, https://arxiv.org/abs/1808.07568.

\bibitem{CR} A. Casta\~{n}eda and G. Robledo,
\emph{Differentiability of Palmer's linearization theorem and converse result for density function},
J. Differ. Equ. {\bf 259} (2015), 4634-4650.

\bibitem{CJ}A. Casta\~{n}eda and N. Jara, \emph{A note on the differentiability of Palmer's topological equivalence for discrete systems}, preprint, https://arxiv.org/abs/2104.14592



\bibitem{CL-JDE95} S. N. Chow and H. Leiva, \emph{Existence and roughness of the exponential dichotomy for skew-product semiflow in Banach spaces},
J. Differ. Equ. \textbf{120} (1995) 429--477.

\bibitem{Cop-book} W. A. Coppel, \emph{Dichotomies in Stability Theory}, Lecture Notes in Math. \textbf{629}, Springer, Berlin, 1978.



\bibitem{CDS} L. V. Cuong, T. S. Doan and S. Siegmund, A \emph{Sternberg theorem for nonautonomous differential
equations}, J. Dynam. Diff. Eq., \textbf{31} (2019), 1279--1299.

\bibitem{DSS} D. Dragi\v cevi\' c, A. L. Sasu and B. Sasu, \emph{On polynomial dichotomies of discrete nonautonomous systems on the half-line}, Carpathian J. Math. \textbf{38} (2022),  663--680.

\bibitem{D} D. Dragi\v cevi\' c, \emph{Admissibility and nonuniform polynomial dichotomies}, Math. Nachr. \textbf{293} (2020), 226--243.

\bibitem{D1} D. Dragi\v cevi\'c, \emph{Global smooth linearization of nonautonomous contractions on Banach spaces}, Electron. J. Qual. Theory Differ. Equ. 12 (2022), 1--19.

\bibitem{DZZ} D. Dragi\v cevi\' c, W. Zhang and W. Zhang, \emph{Smooth linearization of nonautonomous difference equations with a nonuniform dichotomy}, Math. Z. \textbf{292} (2019), 1175--1193.


\bibitem{DZZ20} D. Dragi\v{c}evi\'c, W. N. Zhang and W. M. Zhang, \emph{Smooth linearization of nonautonomous differential equations
with a nonuniform dichotomy}, Proc. London Math. Soc. {\bf 121} (2020), 32--50.

\bibitem{E-book} S. N. Elaydi, {\it An Introduction to Difference Equations}, 3rd ed., Springer, New York, 2005.

\bibitem{Grobman} D. M. Grobman,
 \newblock   {\em  Topological classification of the neighborhood of a singular point in $n$-dimensional space},
 \newblock Mat. Sb. {\bf 56} (1962), 77--94.

\bibitem{HZ-JDE04}  J.K. Hale and W. Zhang, \emph{On uniformity of exponential dichotomies for delay equations}, J. Differ. Equ. \textbf{204} (2004) 1--4.

\bibitem{HartPAMS60} P. Hartman,
 \newblock {\em A lemma in the theory of structural stability of differential equations},
  \newblock    Proc. Amer. Math. Soc. {\bf 11} (1960), 610--620.
  
\bibitem{J} N. Jara, \emph{Smoothness of class $C^2$ of nonautonomous linearization without spectral conditions},  J. Dyn. Diff. Equat. (2022), https://doi.org/10.1007/s10884-022-10207-5


\bibitem{Palis} J. Palis, \emph{On the local structure of hyperbolic points in Banach spaces}, An. Acad. Brasil. Ci\^enc. {\bf 40} (1968), 263-266.

\bibitem{Palmer}
K. Palmer, {\it A generalization of Hartman's linearization theorem}, J. Math. Anal. Appl. {\bf 41} (1973), 753--758.


\bibitem{P-JDE84} K. J. Palmer, \emph{Exponential dichotomies and transversal homoclinic points}, J. Differ. Equ. \textbf{55} (1984) 225--256.

\bibitem{Per-30} O. Perron, \emph{Die stabilit\"atsfrage bei differentialgleichungen}, Math. Z. \textbf{32} (1930) 703--728.

\bibitem{Pot-book} C. P\"otzsche, Geometric theory of discrete nonautonomous dynamical systems, Springer, 2010.


\bibitem{Pugh}
C. Pugh, \emph{On a theorem of P. Hartman},  Amer. J. Math. {\bf 91} (1969), 363-367.

\bibitem{Ray-JDE98} V. Rayskin, {\it $\alpha$-H\"{o}lder linearization}, J. Differ. Equ. {\bf 147} (1998), 271--284.

\bibitem{SS-JDE78} R. J. Sacker and G. R. Sell, \emph{A spectral theory for linear differential systems}, J. Differ. Equ. \textbf{27} (1978), 320--358.


\bibitem{SY-book}    G. R. Sell and Y. You, Dynamics of evolutionary equations \textbf{143}, New York: Springer, 2002.

\bibitem{S-JDEA02} S. Siegmund, \emph{Dichotomy spectrum for nonautonomous differential equations}, J. Differ. Equ. Appl. \textbf{14} (2002), 243--258.

\bibitem{Sternberg2} S. Sternberg,
\newblock   {\em Local contractions and a theorem of Poincar\'e},
 \newblock Amer. J. Math. {\bf 79} (1957),
 809--824.
 
 
\bibitem{Sternberg58} 
 S. Sternberg, \emph{On the structure of local homeomorphisms of Euclidean $n$-space},  Amer. J. Math. \textbf{80} (1958), 623-631.
 
 \bibitem{Stri-JDE90} S. van Strien, {\it Smooth linearization of hyperbolic fixed points without resonance conditions}, J. Differ. Equ. {\bf 85} (1990), 66--90.
 
 \bibitem{ZZJ} W. M. Zhang, W. N. Zhang and W. Jarczyk, {\it Sharp regularity of linearization for $C^{1, 1}$
hyperbolic diffeomorphisms}, Math. Ann. {\bf 358} (2014), 69--113.

\bibitem{ZZ-JFA16} L. Zhou and W. Zhang, \emph{Admissibility and roughness of nonuniform exponential dichotomies for difference equations}, J. Funct. Anal. \textbf{271} (2016), 1087--1129.

\bibitem{ZLZ-JDE17} L. Zhou, K. Lu and W. Zhang, \emph{Equivalences between nonuniform exponential dichotomy and admissibility}, J. Differ. Equ. \textbf{262} (2017), 682--747.

\end{thebibliography}
\end{document}